\documentclass[11pt]{amsart}
\usepackage[T1]{fontenc}

\usepackage{SSdefn, SSams}
\usepackage{eucal, mathtools}
\usepackage{tikz}
\usepackage{tikz-cd}

\newcommand{\FI}{\mathrm{FI}}
\newcommand{\aF}{\mathbb{F}}
\newcommand{\St}{\mathrm{St}}

\newcommand{\VI}{\mathrm{VI}}
\newcommand{\ModVI}{\Mod_{\VI}}
\newcommand{\Modtors}{\Mod_{\VI}^{\tors}}
\newcommand{\VB}{\mathrm{VB}}
\newcommand{\VA}{\mathrm{VA}}
\newcommand{\VIC}{\mathrm{VIC}}
\newcommand{\newshift}{\bar{\Sigma}}
\newcommand{\newdelta}{\bar{\Delta}}

\newcommand{\gen}{\mathrm{gen}}
\newcommand{\dualI}{\check{\cI}}

\newcommand{\Obj}{\mathrm{Obj}}

\newcommand{\bmu}{\boldsymbol{\mu}}

\newcommand{\stacks}[1]{\cite[\href{http://stacks.math.columbia.edu/tag/#1}{Tag~#1}]{stacks}}

\DeclareMathOperator{\Ext}{Ext}

\newcommand{\coloneq}{\mathrel{\mathop:}\mkern-1.2mu=}

\title{$\VI$-modules in non-describing characteristic, part I}

\author{Rohit Nagpal}
\address{Department of Mathematics, University of Chicago, Chicago, IL}
\email{\href{mailto:nagpal@math.uchicago.edu}{nagpal@math.uchicago.edu}}
\urladdr{\url{http://math.uchicago.edu/~nagpal/}}


\begin{document}


\begin{abstract}
Let $\mathrm{VI}$ be the category of finite dimensional $\mathbb{F}_q$-vector spaces whose morphisms are injective linear maps, and let $\mathbf{k}$ be a noetherian ring. We study the category of functors from $\mathrm{VI}$ to $\mathbf{k}$-modules in the case when $q$ is invertible in $\mathbf{k}$. Our results include a structure theorem, finiteness of regularity, and a description of the Hilbert series.  These results are  crucial in the classification of smooth irreducible $\GL_{\infty}(\mathbb{F}_q)$-representations in non-describing characterisitic which is contained in Part II of this paper \cite{VI2}.
\end{abstract}

\maketitle

\tableofcontents

\section{Introduction} 
Fix a commutative noetherian  ring $\bk$.  Set $\aF = \mathbb{F}_q$, and let $\GL_n$ be the $n$th general linear group over $\aF$. Roughly speaking, the aim of this paper is to study the behavior of sequences, whose $n$th member is a $\bk[\GL_n]$-module, as $n$ approaches infinity (the ``generic case''). As $n$ varies, every prime appears as a divisor of the size of $\GL_n$. But surprisingly, it is possible to avoid most of the complications of  modular representation theory in the generic case after inverting just one prime, namely the characteristic of $\aF$.  We assume throughout that $q$ is invertible in $\bk$, and we call this the ``non-describing characteristic'' assumption.

 We obtain these sequences in the form of $\VI$-modules. A $\VI$-module $M$ is a functor \[M \colon \VI \to \Mod_{\bk},\] where $\VI$ is the category of finite dimensional $\aF$-vector spaces with injective linear maps. Clearly, $\GL_n = \Aut_{\VI}(\aF^n)$ acts on $M(\aF^n)$. Thus $M$ can be thought of as a sequence whose $n$th member is a $\bk[\GL_n]$-module. This sequence could be arbitrary if we do not impose any finiteness conditions on $M$. But there is a natural notion of ``finite generation'' in the category of $\VI$-modules. This paper analyzes finitely generated $\VI$-modules.  Here is a sample theorem that we prove (it extends \cite[Theorem~1.7]{gan-watterlond-VI} away from characteristic zero, and also improves some cases of \cite[Corollary~8.3.4]{catgb}):

\begin{theorem}[$q$-polynomiality of dimension]
	\label{intro-polynomiality}
	Assume that $\bk$ is a field in which $q$ is invertible. Let $M$ be a finitely generated $\VI$-module. Then there exists a polynomial $P$  such that $\dim_{\bk} M(\aF^n) = P(q^n)$ for large enough $n$.
\end{theorem}

The result above is a consequence of our main structural result that we prove about finitely generated $\VI$-modules. Given a $\VI$-module $M$ and a vector space $X$, we can define a new $\VI$-module $\Sigma^X M$ by \[\Sigma^X M(Z) = M(X + Z).\] We call this new $\VI$-module the {\bf shift} of $M$ by $X$. Our main result roughly says that the shift of a finitely generated module by a vector space of large enough dimension has a very simple description. To make it precise, note that there is a natural restriction functor \[\ModVI \to  \prod_{n \ge 0} \Mod_{\bk[\GL_n]}. \] This functor admits a left adjoint $\cI$. We call a $\VI$-module {\bf induced} if it is of the form $\cI(W)$ for some $W$. A $\VI$-module that admits a finite filtration whose graded pieces are induced is called {\bf semi-induced}. We now state our main theorem.

\begin{theorem}[The shift theorem]
	\label{intro-shift-theorem}
	Assume that $q$ is invertible in $\bk$. Let $M$ be a finitely generated $\VI$-module. Then $\Sigma^X M$ is semi-induced if the dimension of $X$ is large enough.
\end{theorem}

\subsection{Idea behind the shift theorem} The shift theorem is proven by induction on the degree of generation. To make the induction hypothesis work, we construct a ``categorical derivation'' in the monoidal category of Joyal and Street \cite{joyal-street}. To make it precise, let $\VB$ be the category of finite dimensional $\aF$-vector spaces with bijective linear maps. Joyal and Street considered a monoidal structure\footnote{It is shown in \cite{joyal-street} that this category is actually a braided monoidal category if $\bk$ is a field of characteristic zero. But we don't need the braiding, and so we don't need the characteristic zero assumption} $\otimes_{\VB}$ on $\Mod_{\VB}$ given by  \[ (M\otimes_{\VB} N)(Y)  = \bigoplus_{X \le Y} M(Y/X) \otimes_{\bk} N(X).  \]  We construct a categorical derivation $\newshift$ on $(\Mod_{\VB}, \otimes_{\VB})$. In other words, $\newshift$ satisfies \[\newshift (M \otimes N)  = (\newshift M \otimes N) \bigoplus (M \otimes \newshift N).\]  As pointed out to us by Steven Sam, there is an algebra object $\bA$ in $(\Mod_{\VB}, \otimes_{\VB})$ such that the category of $\VI$-module is equivalent to the category of $\bA$-modules. Under this equivalence, induced modules are $\bA$-modules of the form $\bA \otimes_{\VB} W$. Our categorical derivation shows that if we apply the cokernel of $\id \to \newshift$ to an induced module then we obtain another induced module of strictly smaller degree of generation. This is what makes our inductive proof work. But there is a caveat. Everything said and done in this paragraph so far  is true without any restrictions on the characteristic. On the other hand, the shift theorem is false if we drop the non-describing characteristic assumption. 

The category $\Mod_{\VI}$ naturally contains a localizing subcategory $\Mod_{\VI}^{\tors}$ whose members are called torsion $\VI$-modules. Given a $\VI$-module $M$, we denote the maximal torsion submodule of $M$ by $\Gamma(M)$. The functor $\Gamma$ is left exact, and its right derived functor is denoted $\rR \Gamma$. A crucial technical ingredient in our proof of the shift theorem is the following criterion for semi-induced modules. 

\begin{theorem}
	\label{intro-semi-induced-saturated}
	Assume that $q$ is invertible in $\bk$. Let $M$ be a finitely generated $\VI$-module. Then $M$ is semi-induced if and only if $\rR \Gamma(M) = 0$.
\end{theorem} That a semi-induced $M$ satisfies $\rR \Gamma(M) = 0$ is easy to prove, and doesn't require any assumptions on the characteristic. But the converse requires the non-describing characteristic assumption in two crucial and separate places: (1) $\newshift$ is exact, and (2) $\newshift$ commutes with $\Gamma$. The first one is immediate from our construction of $\newshift$, but the second one requires an interesting combinatorial identity (which appears in the proof of Lemma~\ref{lem:combinatorial-identity}).

The last ingredient of our proof is a recent  theorem proved independently by Putman-Sam \cite{putman-sam} and Sam-Snowden \cite{catgb}  which resolved a long-standing conjecture of Lannes and Schwartz.

\begin{theorem}[{\cite{putman-sam, catgb}}]
	\label{intro:noetherianity}
Suppose $\bk$ is an arbitrary noetherian ring (the non-describing characteristic assumption is not needed). Then the category of $\VI$-modules is locally noetherian.
\end{theorem}

We also need the following immediate corollary of this theorem, which provides us control over the torsion part of a module. 

\begin{corollary}[{\cite{putman-sam, catgb}}]
	\label{intro-cor-noeth}
	Suppose $\bk$ is an arbitrary noetherian ring (the non-describing characteristic assumption is not needed).  Let $M$ be a finitely generated $\VI$-module. Then $\Gamma(M)(X) = 0$ if the dimension of $X$ is large enough.
\end{corollary}

 All these ingredients above allow us to show by induction on the degree of generation that $\newshift^n M$ is semi-induced if $n$ is large enough. The shift theorem then follows from it.

\subsection{Some consequences of the shift theorem}
To start with, Theorem~\ref{intro-polynomiality} is a consequence of the shift theorem simply because induced modules can be easily seen to satisfy $q$-polynomiality of dimension. If we drop the non-describing characteristic assumption, and assume that $\bk =\aF$, then $M(X) = X$ defines a finitely generated $\VI$-module. This implies that $q$-polynomiality fails in equal characteristic, and so the shift theorem must also fail. Below we list some more consequences.

\begin{theorem}[Finiteness of local cohomology]
	\label{intro-finiteness-of-cohomology}
Assume that $q$ is invertible in $\bk$.	Let $M$ be a finitely generated $\VI$-module. Then we have the following: \begin{enumerate}
		\item For each $i$, the module $\rR^i\Gamma(M)$ is finitely generated. In particular, $\rR^i\Gamma(M)(X) = 0$ if the dimension of $X$ is large enough.
		\item $\rR^i\Gamma(M) = 0$ for $i$ large enough.
	\end{enumerate}
\end{theorem}

The theorem above extends Corollary~\ref{intro-cor-noeth} to the higher derived functors of $\Gamma$. We use this theorem, and an argument similar to the one for $\FI$-modules as in \cite{regthm}, to bound the regularity. In particular, we  provide a bound on the regularity in terms of the degrees of the local cohomology.

\begin{theorem}[Finiteness of regularity]
	\label{intro:regularity}
Assume that $q$ is invertible in $\bk$.	Let $M$ be a finitely generated $\VI$-module. Then $M$ has finite Castelnuovo-Mumford regularity.
\end{theorem}

Gan and Watterlond have shown in \cite{gan-watterlond-VI} that, when $\bk$ is an algebraically closed field of characteristic zero, then any finitely generated $\VI$-module exhibits ``representation stability'', a phenomenon described by Church and Farb in \cite{rep-stability}. Representation stability for $\VI$-modules also follows from a recent result of Gadish \cite[Corollary~1.13]{nir}. We prove representation stability in a more systematic way. We believe that our method  can be used to write down a virtual specht stability statement away from characteristic zero as done for $\FI$-modules by Harman in \cite{virtual-stab}. In contrast to this, the methods in \cite{gan-watterlond-VI} or \cite{nir} use characteristic zero assumption in an essential way. Below, we only state a part of the result to avoid giving a full definition of representation stability here (for full definition, see \ref{sec:rep-stability}).

\begin{theorem}[{\cite[Theorem~1.6]{gan-watterlond-VI}}]
	\label{intro:rep-stability}
	Assume that $\bk$ is an algebraically closed field of characteristic zero. Let $M$ be a finitely generated $\VI$-module. Then the length of the $\bk[\GL_n]$-module $M(\aF^n)$ stabilizes in $n$.
\end{theorem}

We also obtain the following new  theorem in characteristic zero.

\begin{theorem}[Finiteness of Injective dimension]
	Assume that $\bk$ is a field of characteristic zero. Then the following holds in $\ModVI$: \begin{enumerate}
		\item Every projective is injective.
		\item Every torsion-free injective is projective.
		\item Every finitely generated module has finite injective dimension.
	\end{enumerate}
\end{theorem}

Along the way, we classify all indecomposable injectives in characteristic zero, and we also classify indecomposable torsion injectives when $\bk$ is an arbitrary noetherian ring.

\subsection{Relations to other works}
Recently, Kuhn in \cite{kuhn-VA} has analyzed a similar but simpler (of lower krull dimension) category of $\VA$-modules, where $\VA$ is the category of finite dimensional $\aF$ vector spaces. 

\begin{theorem}[{\cite[Theorem~1.1]{kuhn-VA}}]
	In the non-describing characteristic, $\Mod_{\VA}$ is equivalent to the product category $\prod_{n \ge 0} \Mod_{\bk[\GL_n]}$. In particular, if $\bk$ is a field then $\Mod_{\VA}$ is of krull dimension zero.
\end{theorem}

It is a folklore that one recovers the representation theory of the symmetric groups from the representation theory of the finite general linear group over $\mathbb{F}_q$ by setting $q=1$. We observe a similar phenomenon between $\FI$-modules and $\VI$-modules: all the results we have for $\VI$-modules in the non-describing characteristic are true for $\FI$-modules in all characteristic ($\FI$-modules encode sequences of representations of the symmetric groups; see \cite{fimodules}). In other words, the proofs for the results on $\FI$-modules are degenerate cases of the proofs for the corresponding results on $\VI$-modules in the non-describing characteristic. But we point out that (1) many of our ideas are copied from the corresponding ideas on $\FI$-modules, and (2) we know a lot more about $\FI$-modules, for example, all the questions that we  pose below have been solved for $\FI$-modules. We have tried to summarize  throughout the text where each crucial idea has been borrowed from, but here is a list of references that contain analogs of our results -- \cite{church-bounding, castelnuovo-regularity, fimodules, fi-noeth, djament, djament-vespa, li, li-ramos, nagpal, regthm, ramos, symc1}.

A higher dimension category of similar representation theoretic nature whose structure is well understood is the category of $\FI_d$-modules; see \cite{symd}, \cite{fidreg}.

\subsection{Further comments and questions}  Theorem~\ref{intro:rep-stability} implies that every finitely generated object in the category \[\ModVI^{\gen} \coloneq \ModVI/\ModVI^{\tors}\] of generic $\VI$-modules is of finite length, that is, the krull dimension of $\ModVI^{\gen}$ is zero. In a subsequent paper \cite{VI2}, we shall prove that the same holds in the non-describing characteristic (where $\bk$ is still assumed to be a field) by providing a complete set of irreducibles of the generic category. Determining Krull dimension in equal characteristic ($\bk = \aF$) is related to an old open problem called the strong artinian conjecture \cite{powell, powell2}.

Sam and Snowden have proven that the categories of  torsion and the generic $\FI$-modules are equivalent in characteristic zero \cite[Theorem~3.2.1]{symc1}, and such a phenomenon seem to appear in some other categories as well (for example, see \cite{infrank} and \cite{sym2noeth} for the category of $\Sym(\Sym^2)$-modules). We have the following question along the same lines:

\begin{question}
	Assume that $\bk$ is of characteristic zero. Is there an equivalence of categories  $\ModVI^{\tors} \cong \ModVI^{\gen}$.
\end{question}

\begin{remark}
	After the release of the first draft of this paper, Gan, Li and Xi have positively answered the question above; see \cite{GLX}.  We note that they used the shift theorem (Theorem~\ref{intro-shift-theorem}) non-trivially; see \cite[Lemma~4.1]{GLX}.
\end{remark}

Our result provides bounds on the Castelnuovo-Mumford regularity in terms of the local cohomology. But we have not been able to bound local cohomology in terms of the degrees of generation and relation. An analogous question for $\FI$-modules has already been answered (\cite[Theorem~A]{castelnuovo-regularity}); also see   \cite{church-bounding}, \cite{li}, and \cite[Theorem~E]{li-ramos} for more results on this. We also note that, in characteristic zero, Miller and Wilson have provided bounds on the higher syzygies for a similar category called $\VIC$-modules; see \cite[Theorem~2.26]{quant-rep-VIC}.

\begin{question}
	\label{ques:gamma-is-finitely-supported}
	Let $M$ be a $\VI$-module generated in degrees $\le t_0$ and whose syzygies are generated in degrees $\le t_1$. Is there a number $n$ depending only on $t_0$ and $t_1$ such that $\Gamma(M)(X) = 0$ for every vector space $X$ of dimension larger than $n$. 
\end{question}

\begin{remark}
	After the release of the first draft of this paper, Gan and Li have positively answered the question above; see \cite{GL}.  We note that they used the shift theorem (Theorem~\ref{intro-shift-theorem}) non-trivially. Along the way, they also made all the bounds in the current paper explicit in terms of $t_0$ and $t_1$; see \cite[Theorem~1.1]{GL}. Bounds in the current paper are in terms of degrees of the local cohomology groups.
\end{remark}

The question below is a $\VI$-module analog of \cite[Conjecture~1.3]{li-ramos} which has been resolved for $\FI$-modules in \cite{regthm}.

\begin{question}
	\label{ques:regularity}
	Is it true that the Castelnuovo-Mumford regularity of a $\VI$-module is exactly $\max_{i} (\deg \rR^i \Gamma(M) + i)$ where $i$ varies over the finitely many values for which $ \rR^i \Gamma(M) $ is nonzero? 
\end{question}

\subsection{Outline of the paper}
In \S \ref{sec:overview}, we provide an overview of $\VI$-modules. In particular, we sketch an equivalence between $\ModVI$ and the module category of an algebra object $\bA$ in the monoidal category of Joyal and Street, and we recall some formalism of local cohomology and saturation from \cite{symd}. In \S \ref{sec:semi-induced}, we prove some formal properties of induced and semi-induced modules that we need. These properties are formal in the sense that they have nothing much to do with $\VI$-modules and are true (with appropriate definitions) in several other categories (for example, $\Mod_{\FI}$, $\Mod_{\FI_d}$ or $\Mod_{\VIC}$). We decided to include a short section and collect these formal results at one place. The meat of the paper is contained in \S \ref{sec:shift} where we prove the shift theorem. The last section (\S \ref{sec:consequences}) contains all the consequences of the shift theorem.

\subsection{Acknowledgments} We would like to thank Steven Sam for pointing out connections between the $\VI$-modules and the category of Joyal and Street \cite{joyal-street}. This helped us obtain a clean proof of Proposition~\ref{prop:categorical-derivation}. We thank Nate Harman for useful conversations on virtual specht stability. We thank Nir Gadish, Liping Li, Peter Patzt, Eric Ramos, Steven Sam, Andrew Snowden and an anonymous referee for pointing out several errors and suggesting numerous improvements on the first draft. 

We thank Geoffrey Powell for pointing out an error  in the published version of this paper and allowing us to use his argument to fix the error (see errata after Corollary~\ref{cor:short-exact-sequence-semi-induced}).

\section{Overview of $\VI$-modules} 
\label{sec:overview}
\subsubsection*{Notation} We work over a unital commutative ring $\bk$. For a non-negatively graded $\bk$-module $M$, we define $\deg M$ to be the least integer $n \ge -1$ such that  $M_k =0$ for $k > n$, and $\deg M = \infty$ if no such $n$ exists. 

We fix a finite field $\aF$ of cardinality $q$, and assume that all vector spaces are over $\aF$. For a vector space $X$, we denote the group of automorphisms of $X$ by $\Aut(X)$ or $\GL(X)$. When the dimension of $X$ is $n$, we also denote these groups by $\GL_n$. We denote the trivial vector space by $0$, and we shall simply write $X \preceq Y$ whenever $\dim_{\aF} X \le \dim_{\aF} Y$. 

\subsection{The monoidal category of Joyal and Street}  We denote, by $\VB$, the category of finite dimensional vector spaces with isomorphisms. A $\VB$-module is a functor from $\VB$ to $\Mod_{\bk}$. $\VB$-modules form a category $\Mod_{\VB}$ which is naturally equivalent to the product category $\prod_{n \ge 0} \Mod_{\bk[\GL_n]}$.  In particular, a $\VB$-module is naturally a non-negatively graded $\bk$-module. We denote, by $V_d$, the $\VB$-module satisfying \begin{align*}
  V_d(X)= \begin{cases} 
 V(X) &\mbox{if} \dim_{\aF} X = d\\
 0 &\mbox{if} \dim_{\aF} X \neq d.
\end{cases}
\end{align*} If $V = V_d$, we say that $V$ is supported in degree $d$. Given $\VB$-modules $M, N$ we define an external product $\otimes_{\VB}$ by \[ (M\otimes_{\VB} N)(Y)  = \bigoplus_{X \le Y} M(Y/X) \otimes_{\bk} N(X).  \] Then $\otimes_{\VB}$ turns $\Mod_{\VB}$ into a monoidal category; see \cite[\S 2]{joyal-street}.

\subsection{The algebra $\bA$} Let $\bA$ be the $\VB$-module such that $\bA_n = \bk$ is the trivial representation of $\GL_n$ for each $n$. We have a map $\bA \otimes_{\VB} \bA \to \bA$ given  by \[ a \otimes b \in \bA(Y/X) \otimes_{\bk} \bA(X) \mapsto ab \in \bA(Y). \] This turns $\bA$ into an algebra object in the monoidal category $(\Mod_{\VB}, \otimes_{\VB} )$. We denote the category of $\bA$-modules by $\Mod_{\bA}$. The $\VB$-module $\bk = \bA/\bA_+$ is naturally an $\bA$-module. As usual, the {\bf degree of generation} of an $\bA$-module $M$ is defined to be $\deg \bk \otimes_{\bA} M$. We shall denote $\deg \Tor_i^{\bA}(\bk, M)$ by $t_i(M)$, and so the degree of generation of $M$ is $t_0(M)$. We say that an $\bA$-module is {\bf presented in finite degrees} if $t_0(M)$ and $t_1(M)$ are finite.

\subsection{Definition of a $\VI$-module}
We denote, by $\VI$, the category of finite dimensional vector spaces with injective linear maps. A $\VI$-module is a functor from $\VI$ to $\Mod_{\bk}$. We denote the category of $\VI$-modules by $\ModVI$. Let $M$ be a $\VI$-module. A $\VI$ morphism $f \colon X \to Y$ induces a map $ M(X) \to M(Y)$ which we denote by $f_{\star}$.  The $\VI$-module $M$ restricts to a $\VB$-module and admits a natural map $\bA \otimes_{\VB} M \to M$ given by \[ a \otimes b \in  \bA(Y/X) \otimes_{\bk} M(X) \mapsto a \iota_{\star}(b) \in  M(Y)  \] where $\iota \colon X \to Y$ is the inclusion. Conversely, if $M$ is an $\bA$-module and $f \colon X \to Y$ is a $\VI$-morphism, then we have a map $f_{\star} \colon M(X) \to M(Y)$ given by the composite \[ M(X) \to M(f(X))  \xrightarrow{1 \otimes -} A(Y/f(X)) \otimes_{\bk} M(f(X))  \to M(Y)   \] where the first map comes from $\VB$-module structure on $M$ and the last map comes from $\bA$-module structure on $M$. It is easy to see that the above discussion describes an equivalence of categories:

\begin{proposition}
$\ModVI$ is equivalent to $\Mod_{\bA}$. 
\end{proposition}

We shall not distinguish between $\VI$-modules and $\bA$-modules. In particular, notions like degree of generation makes sense for $\VI$-modules. We explain degree of generation from the $\VI$ perspective now. Given a $\VB$-module $V$, we can upgrade it to a $\VI$-module by declaring that all $\VI$-morphisms that are not isomorphisms acts on $V$ by $0$. This defines a functor $\Psi^{\uparrow} \colon \Mod_{\VB} \to \ModVI$. We define $\rH_0^{\VI}$ to be the left adjoint to $\Psi^{\uparrow}$.  Let $M$ be a $\VI$-module. Denote the smallest $\VI$-submodule containing $M(Y)$ for $Y \prec X$ by $M_{\prec X}$. Then $\rH_0^{\VI}(M)$ is given explicitly by \[ \rH_0^{\VI}(M)(X) = (M/M_{\prec X})(X). \] The functor $\rH_0^{\VI}$ (called {\bf $\VI$-homology}) is same as the functor $\Tor_0^{\bA}(\bk, - )  = \bk \otimes_{\bA} - $ under the equivalence above. We shall use the notation $\rH_i^{\VI}(-)$ instead of $\Tor_i^{\bA}(\bk, -)$. Here are some basic results on $\VI$-homology.

\begin{proposition}
	\label{prop:homology-filtration}
	We have $\rH^{\VI}_0(M_{\prec d}) = \rH^{\VI}_0(M)_{< d}$. In particular, if $n < m$ then the natural map $\rH^{\VI}_0(M_{\prec n})  \to \rH^{\VI}_0(M_{\prec m})$ is just the inclusion map $\rH^{\VI}_0(M)_{< n} \to \rH^{\VI}_0(M)_{<m}$.
\end{proposition}
 
\begin{proposition}
\label{prop:nakayama}
Let $M$ be a $\VI$-module, and $f \colon M \to N$ be a morphism of $\VI$-modules. Then we have the following \begin{enumerate}
	\item $\rH_0^{\VI}(M) = 0$ if and only if $M=0$.
	\item $\rH_0^{\VI}(f)$ is an epimorphism if and only if $f$ is an epimorphism.
	\item Suppose $t_0(M) \le d$ and $N(X) = 0$ for $X \prec \aF^d$. Then $\rH^{\VI}_0(f) = 0$ if and only if $f=0$.
\end{enumerate} 
\end{proposition}
\begin{proof} Part (a) is just the Nakayama lemma, and (b) follows from (a) and the right exactness of $\rH_0^{\VI}$. For part (c), suppose $\rH^{\VI}_0(f) = 0$. By part (a), it suffices to show that $\rH^{\VI}_0(\im f) =0$. First suppose $X$ is a vector space of dimension at most $d$. Since $N(Y) = 0$ for all $Y \prec X$, the map $f(X) \colon M(X) \to N(X)$ factors through the projection $M(X) \to \rH^{\VI}_0(M)(X)$ and $N(X)$ is naturally isomorphic to $\rH^{\VI}_0(N)(X)$. This shows that \[\rH^{\VI}_0(\im f)(X) = (\im \rH^{\VI}_0(f))(X) = 0. \] Next suppose $X$ is a vector space of dimension bigger than $d$. Since $M \to \im f$ is a surjection and $\rH^{\VI}_0$ is right exact we see that $t_0(M) \le d \Longrightarrow \rH^{\VI}_0(\im f)(X) = 0$. Thus $\rH^{\VI}_0(\im f) =0$, completing the proof.
\end{proof}

\subsection{Local cohomology and saturation}
\label{subsec:local-cohomology-preliminaries}
Let $M$ be a $\VI$-module. We say that an element $x \in M(X)$ is {\bf torsion} if there exists an injective linear map $f \colon X \to Y$ such that $f_{\star}(x) =0$. A $\VI$-module is  torsion if it consists entirely of torsion elements. We denote the maximal torsion submodule of $M$ by $\Gamma(M)$, the $i$th right derived functor of $\Gamma$ by $\rR^i \Gamma$, and the degree of $\rR^i \Gamma(M)$ by $h_i(M)$. Let $\Modtors$ be the category of torsion $\VI$-modules.  It is easy to see that $\Modtors \subset \ModVI$ is a localizing subcategory. Let  $\rT \colon \ModVI \to \ModVI/\Modtors$ be the corresponding localization functor and $\rS$ be its right adjoint (the section functor).  We define {\bf saturation} of $M$ to be the composition $\bS(M) = \rS\rT(M)$. We denote the $i$th right derived functor of $\bS$ by $\rR^i \bS$.

We refer the readers to \cite[\S 4]{symd} where the formalism of local cohomology and saturation is discussed in quite generality. This formalism needed an assumption which in our case is the following: \begin{itemize}
\item[(*)] Injective objects of $\Modtors$ remain injective in $\ModVI$. 
\end{itemize} 

We note here that both $\Modtors$ and $\ModVI$ are Grothendieck abelian categories, and so both contain enough injectives.

\begin{lemma}
	\label{lem:injective-hull-new}
	The assumption (*), as above, holds. In particular, the injective hull (as $\VI$-modules) of a torsion module is torsion.
\end{lemma}
\begin{proof}
The first assertion follows immediately from Theorem~\ref{intro:noetherianity} and \cite[Proposition~4.18]{symd}. Now suppose that $M$ is a torsion $\VI$-module. Then we can embed $M$ into an injective object $I$ in $\ModVI^{\tors}$. By (*), $I$ is injective in $\ModVI$, and so $I$ contains the injective hull of $M$. The second assertion is immediate from this.
\end{proof}


\begin{lemma} \label{lem:torsion-injective}
If $I$ is injective in $\ModVI$, then $\Gamma(I)$ is also injective in $\ModVI$. In particular, if $M$ is a torsion $\VI$-module, then $\rR^i \Gamma(M) = 0$ for $i>0$.
\end{lemma}
\begin{proof} 
Since $I$ is injective and contains $\Gamma(I)$, it follows that $I$ contains the injective hull of $\Gamma(I)$. By the previous lemma and the maximality of $\Gamma(I)$, we conclude that $\Gamma(I)$ is its own injective hull. This proves the first assertion.

 The first assertion implies that if $M$ is a torsion module then it admits an injective resolution $M \to I^{\bullet}$ such that each $I^i$ is torsion. Since $\Gamma$ is identity on torsion modules, we see that $\Gamma(I^{\bullet}) = I^{\bullet}$. The second assertion follows.
\end{proof}

\begin{corollary}
	\label{cor:torsion-have-no-saturation}
	Let $T$ be an object of the right derived category  $\rD^{+}(\ModVI)$ which can be represented by a complex of torsion $\VI$-modules. Then $\rR \Gamma (T) \cong T$, and $\rR \bS (T) = 0$.
\end{corollary}

We now state a result from \cite{symd} that we need.

\begin{proposition}[{\cite[Proposition~4.6]{symd}}]
	\label{prop:triangle}
	Let $M \in \rD^{+}(\ModVI)$. Then we have an exact triangle  \[\rR \Gamma(M) \to M \to \rR \bS(M) \to \] where the first two maps are the canonical ones.
\end{proposition} 

We call a $\VI$-module $M$ {\bf derived saturated} if $M \to \rR \bS(M)$ is an isomorphism in $\rD^{+}(\ModVI)$, or equivalently $\rR \Gamma(M) = 0$ (see the proposition above).

\section{Induced and semi-induced $\VI$-modules}
\label{sec:semi-induced}

The aim of this section is to prove some formal properties of induced and semi-induced modules. The restriction map $\Psi_{\downarrow} \colon \ModVI \to \Mod_{\VB}$ admits a left adjoint $\Mod_{\VB} \to \ModVI$ denoted $\cI$, which is exact. By definition of $\cI$, we have the adjunction \begin{align}
\Hom_{\ModVI}(\cI(V), M) = \Hom_{\Mod_{\VB}}(V, M). \tag{*}
\end{align} We call $\VI$-modules of the form $\cI(V)$ {\bf induced}. If $V$ is supported in degree $d$ we say that $\cI(V)$ is induced from degree $d$. Moreover, when $V_d$ is a $\VB$-module isomorphic to $\bk[\Hom_{\VB}(\aF^d, -) ]$ then we denote $\cI(V)$ by simply $\cI(d)$. By Yoneda lemma, we have $\cI(d) = \bk[\Hom_{\VI}(\aF^d, -)]$. We have the following alternative descriptions for $\cI(V)$: \begin{align*}
\cI(V) & = \bA \otimes_{\VB} V, \\
\cI(V) &= \bigoplus_{d \ge 0} \cI(d) \otimes_{\bk[\Aut(\aF^d)]}  V(\aF^d).
\end{align*}

\begin{proposition} 
\label{prop:HI}	
The composite functor $\rH^{\VI}_0 \cI$ is naturally isomorphic to the identity functor on $\VB$-modules. The counit $\cI \Psi_{\downarrow} \to \id$ is an epimorphism on any $\VI$-module.
\end{proposition}
\begin{proof} The first assertion is clear because composing $\bk \otimes_{\bA} - $ with $\bA \otimes_{\VB} -$ yields $\bk \otimes_{\VB} -$, which is naturally isomorphic to the identity functor. Alternatively, by adjointness of $\cI$ and $\rH^{\VI}_0$, we have \[\Hom_{\Mod_{\VB}}(\rH^{\VI}_0 \cI(M), N) = \Hom_{\Mod_{\VB}}(M, \Psi_{\downarrow}\Psi^{\uparrow}N) = \Hom_{\Mod_{\VB}}(M, N),\] and so the result follows by the uniqueness of left adjoints. For the second assertion, it suffices to check that $\Psi_{\downarrow}$ is faithful, which is trivial.
\end{proof}

A useful thing to note is that if $M$ is a $\VI$-module and $f \colon V \to M$ is a map of $\VB$-modules then the image of the corresponding map $g \colon \cI(V) \to M$ is the smallest $\VI$-submodule of $M$ containing the image of $f$. In particular, if $V(X) \to M(X)$ is surjective then $g(X)$ is surjective.

\begin{proposition}
\label{prop:projectives}
 $\cI(V)$ is a projective $\VI$-module if and only if $V$ is a projective $\VB$-module. All projective $\VI$-modules are of the form $\cI(V)$.
\end{proposition}
\begin{proof} Each of  $\cI$ and $\rH^{\VI}_0$ is left adjoint to an exact functor ($\Psi_{\downarrow}$ and $\Psi^{\uparrow}$ respectively), so both of them preserve projectives (\cite[Proposition~2.3.10]{weibel}). Since $\rH^{\VI}_0 \cI = \id$ (Proposition~\ref{prop:HI}), we conclude that $\cI(V)$ is projective if and only if $V$ is projective.
	
For the second assertion, let $P$ be a projective $\VI$-module. By Proposition~\ref{prop:HI}, there is a natural surjection $\phi \colon \cI \Psi_{\downarrow}(P) \to P $, and since $P$ is projective it admits a section $s$. Let $\psi \colon \cI \rH^{\VI}_0 (P) \to P$ be the map given by $\psi = \phi \circ \cI \rH^{\VI}_0 (s)$.  It suffices to show that $\psi$ is an isomorphism. By Proposition~\ref{prop:HI}, we have \[\rH^{\VI}_0 (\psi) \cong \rH^{\VI}_0 (\phi \circ s) = \rH^{\VI}_0(\id) =\id.\] Thus, by Proposition~\ref{prop:nakayama}, $\psi$ is surjective. Since $P$ is projective we have a short exact sequence \[ 0 \to \rH^{\VI}_0(\ker \psi) \to \rH^{\VI}_0(\cI \rH^{\VI}_0 (P))  \xrightarrow{\rH^{\VI}_0 (\psi)  \cong \id}  \rH^{\VI}_0 (P) \to 0.  \] In particular, $\rH^{\VI}_0(\ker \psi)=0$. Thus, by Proposition~\ref{prop:nakayama}, we conclude that $\psi$ is an isomorphism. This completes the proof.
\end{proof}

\begin{corollary}
	$\ModVI$ has enough projectives.
\end{corollary}
\begin{proof} Clearly, $\Mod_{\VB} \cong \prod_{n \ge 0} \Mod_{\bk[\GL_n]}$ has enough projectives. Now let $M$ be a $\VI$-module and let $P \to \Psi_{\downarrow}(M)$ be a surjection from a projective $\VB$-module $P$. Then, the composite $\cI(P) \to \cI \Psi_{\downarrow}(M) \to M$ is a surjection (Proposition~\ref{prop:HI}) and $\cI(P)$ is projective (Proposition~\ref{prop:projectives}), completing the proof.
\end{proof}

\begin{proposition}
	\label{prop:induced-homology-acyclic}
 $\rH^{\VI}_i (\cI(V)) = 0$ for $i > 0$ and is isomorphic to 	$V$ for $i=0$. In particular, $t_0(\cI(V)) = \deg V$, and $\cI(V)$ is presented in finite degrees if and only if $\deg(V) < \infty$.
\end{proposition}
\begin{proof}
Let $P_{\bullet} \to V$ be a projective resolution of $V$ as a $\VB$-module. Then $\cI(P_{\bullet})$ is a projective resolution of $\cI(V)$ (Proposition~\ref{prop:projectives}). The assertion now follows by applying $\rH_0^{\VI}(-)$ and noting that $\rH_0^{\VI} \cI = \id$ (Proposition~\ref{prop:HI}).
\end{proof}

\begin{proposition}
	\label{prop:d-equivalence}
	Let $\cI(U), \cI(V)$ be $\VI$-modules induced from $d$. Then $\rH^{\VI}_0$ induces an isomorphism \[\Hom_{\ModVI}(\cI(U),  \cI(V))  \to \Hom_{\Mod_{\VB}}(U,V),\] whose inverse is given by $\cI$.  
\end{proposition}
\begin{proof}
 By Proposition~\ref{prop:HI}, $\rH^{\VI}_0 \cI = \id$. Conversely, suppose $f \in \Hom_{\ModVI}(\cI(U),  \cI(V))$. Then, again by Proposition~\ref{prop:HI}, $\rH^{\VI}_0(f - \cI \rH^{\VI}_0(f)) = 0$. Thus, by Proposition~\ref{prop:nakayama} (3), we conclude that $f - \cI \rH^{\VI}_0(f) = 0$, completing the proof.
\end{proof}

\begin{proposition} 
	\label{prop:induced-from-d}
	Kernel and cokernel of a map of $\VI$-modules induced from $d$ are induced from $d$. An extension of $\VI$-modules induced from $d$ is induced from $d$. 
\end{proposition}
\begin{proof} Let $f \colon \cI(U) \to \cI(V)$ be a map of $\VI$-modules. Then by the previous proposition, there is a $g \colon U \to V$ such that $f = \cI(g)$. Since $\cI$ is exact, we have $\ker f = \cI(\ker g)$ and $\coker f = \cI(\coker g)$, proving the first assertion. For the second assertion, let $M$ be an extension of $\cI(U)$ and $\cI(V)$. Let $P_{\bullet} \to U$ and $Q_{\bullet} \to V$ be  projective resolutions of $U$ and $V$ such that $P_i$ and $Q_i$ are all supported in degree $d$. By the horseshoe lemma and Proposition~\ref{prop:projectives}, $\cI(P_{\bullet}+Q_{\bullet})$ is a projective resolution of $M$. By the first assertion, $M$ is induced from $d$.
\end{proof}

\begin{proposition}
	\label{prop:induced-generated-in-d}
	Let $\cI(W)$ be a module induced from $d$. And let $M$ be a submodule of $\cI(W)$ generated in degrees $\le d$. Then $M$ is isomorphic to $\cI(M_d)$. In particular, $M$ is induced from $d$. 
\end{proposition}
\begin{proof}  Since $M$ is generated in degree $d$ and $M_k \subset \cI(W)_k = 0$ for $k<d$, we have $\rH^{\VI}_0(M) = M_d$. It follows that the natural map $f \colon \cI(M_d) \to M$ is a surjection (Proposition~\ref{prop:nakayama}). Composing it with the inclusion $M \to \cI(W)$, we obtain a map $g \colon \cI(M_d) \to \cI(W)$.  By construction, $\rH^{\VI}_0(g)$ is the natural inclusion $M_d \to W$. Thus by the Proposition~\ref{prop:d-equivalence}, we have \[\ker(g) = \ker(\cI \rH^{\VI}_0(g)) = \cI (\ker(\rH^{\VI}_0(g)) )= \cI(0) = 0.  \] This implies that  $f$ is injective, completing the proof.
\end{proof}

\begin{proposition} 
	\label{prop:generation-presentation}
	Let $M$ be a $\VI$-module. Then \begin{enumerate}
	\item $M$ is generated in degrees $\le d$ if and only if it admits a surjection $ \cI(V) \to M$ with $\deg V \le d$.  
	\item $M$ is presented in finite degrees if and only if there is an exact sequence \[ \cI(W) \to \cI(V) \to M \to 0 \] such that $\deg V, \deg W < \infty$.
\end{enumerate} 
\end{proposition}
\begin{proof}
Proof of (a). Suppose there is a surjection $\cI(V) \to M$. Since $\rH^{\VI}_0$ is right exact, we have a surjection $V \to \rH^{\VI}_0(M)$. This shows that $\deg V \le d \Longrightarrow t_0(M) \le d$. Conversely, suppose $t_0 \le d$. Let $V$ be the $\VB$-module with $\deg V \le d$ satisfying $V(X) = M(X)$ for $\dim X \le d$. By construction, we have a surjection $V \to \rH^{\VI}_0(M)$. By Nakayama lemma, the natural map $\cI(V) \to M$ is a surjection, completing the proof.

Proof of (b). First suppose $M$ is presented in finite degrees. Then by part (a), there is a surjection $f \colon \cI(V) \to M$ with $\deg V < \infty$. It suffices to show that the kernel of $f$ is generated in finite degrees. But this follows from the long exact sequence corresponding to $\rH^{\VI}_0$. Conversely, if there is an exact sequence \[ \cI(W) \to \cI(V) \to M \to 0 \] such that $\deg V, \deg W < \infty$. Then by part (a), $M$ and the kernel of $ \cI(V) \to M$ are generated in finite degrees. Again, the long exact sequence corresponding to $\rH^{\VI}_0$ finishes the proof (see Proposition~\ref{prop:induced-homology-acyclic}).
\end{proof}

\subsection{Semi-induced modules}

We call a module {\bf semi-induced} if it admits a finite filtration whose graded pieces (successive quotients) are induced modules that are generated in finite degrees.

\begin{lemma} 
	\label{lem:ramos-single-new}
	Suppose $\rH^{\VI}_1(Q) =0$ and  assume that $\rH^{\VI}_0(Q)$ is concentrated in degree $d$. Then $Q$ is induced from $d$. In particular, $Q$ is homology acyclic.
\end{lemma}
\begin{proof}
	By the assumption, $Q_d = \rH^{\VI}_0(Q)$. This implies that there is a natural surjection $\phi \colon M \coloneq \cI(\rH^{\VI}_0(Q)) \to Q$ which induces an isomorphism $\rH^{\VI}_0(M) \to \rH^{\VI}_0(Q)$. By the assumption that $\rH^{\VI}_1(Q) =0$ and the nakayama lemma, we see that the kernel of $\phi$ is trivial. This shows that $Q$ is induced from $d$. The statement that $Q$ is homology acyclic follows from Proposition~\ref{prop:induced-homology-acyclic}. 
\end{proof}

The proof of the following proposition is motivated by a very similar theorem of Ramos for $\FI$-modules \cite[Theorem~B]{ramos}.

\begin{proposition}
	\label{prop:ramos-new} Let $M$ be a module generated in finite degrees. Then $M$ is homology acyclic if and only if $M$ is semi-induced. More generally, if $\rH^{\VI}_1(M) = 0$ then the graded pieces (successive quotients $Q^i \coloneq M_{\preceq i}/M_{\prec i}$) of the natural filtration  \[0 \subset M_{\preceq 0} \subset \ldots \subset M_{\preceq d} =M\] are induced (more precisely, $Q^i$ is induced from $i$).
\end{proposition}
\begin{proof}
	By Proposition~\ref{prop:induced-homology-acyclic}, if $M$ is semi-induced then it satisfies $\rH^{\VI}_i(M) =0$ for $i>0$, and is thus acyclic. The reverse inclusion follows from the second assertion which we now prove by induction on $d \coloneq t_0(M)$. Note that $\rH^{\VI}_0(Q^i)$ is concentrated in degree $i$, and  $\rH^{\VI}_0(M_{\prec d})$ injects into $\rH^{\VI}_0(M_{\preceq d})$ (Proposition~\ref{prop:homology-filtration}). Thus applying $\rH^{\VI}_0(-)$ to the exact sequence \[0 \to M_{\prec d} \to M \to Q^d \to 0\] shows that $\rH^{\VI}_1(Q^d) = 0$. By Lemma~\ref{lem:ramos-single-new}, $Q^d$ is induced from $d$, and hence acyclic. Thus	$\rH^{\VI}_1(M_{\prec d}) = 0$. The rest follows by induction.
\end{proof}

\begin{corollary}
	\label{cor:good-filtration-new}
	Suppose $M$ is semi-induced module generated in degree $\le d$. Then the graded pieces (successive quotients $Q^i \coloneq M_{\preceq i}/M_{\prec i}$) of the natural filtration  \[0 \subset M_{\preceq 0} \subset \ldots \subset M_{\preceq d} =M\] are induced (more precisely, $Q^i$ is induced from $i$). 
\end{corollary}


%
%
%
\section{The shift theorem}
\label{sec:shift}

The aim of this section is to prove our main result -- the shift theorem.

\subsection{The shift and the difference functors I} \label{subsec:shift-difference-1}
The category of $\aF$-vector spaces (and in particular, $\VI$) has a symmetric monoidal structure $+$ given by the direct sum of vector spaces.  It allows us to define a shift functor $\tau^X$ on $\aF$-vector spaces (or on $\VI$) by \[\tau^X(Z) = X + Z.\] Moreover, for any $\aF$-linear map $\ell \colon X \to Y$, we have a natural transformation $\tau^{\ell} \colon \tau^X \to \tau^Y$ given by $\tau^{\ell}(Z) = \ell + \id_Z$.

We say that a morphism $f \colon \aF^d \to X + Z$ is of {\bf $X$-rank $k$} if the dimension of $(X + \im f)/X$ is $k$ (clearly, $X$-rank of $f$ is at most $d$). In other words, $k$ is the least integer such that there are $\VI$-morphisms $g \colon \aF^d \to X + \aF^k$  and $h \colon \aF^k \to Z$ satisfying $f = \tau^X(h) g$. We call any decomposition of the form $f = \tau^X(h) g$ as above, an {\bf $(X,k)$-decomposition} of $f$. The following lemma is immediate from basic linear algebra.

\begin{lemma}
Let $\tau^X(h_1) g_1 = \tau^X(h_2) g_2$ are two $(X,k)$-decompositions of $f \colon \aF^d \to X + Z$. Then there is a unique $\sigma \in \GL_k$ such that $g_2 = \tau^X(\sigma)g_1$ and $h_2 = h_1 \sigma^{-1}$.
\end{lemma}

Let $D^d_k(X, Z)$ be the free $\bk$-module on morphisms $f \colon \aF^d \to X + Z$ of $X$-rank $k$. Then $D^d_k(X, Z)$ is a $\VI$-module in both of the arguments $X$ and $Z$, and has a natural action of $\GL_d$ on the right.

\begin{lemma} 
\label{lem:rank-decomposition}	
	We have the following: \begin{enumerate}
		\item $D^d_k(X, \aF^k)$ is a free $\bk[\GL_k]$-module.
		\item $D^d_k(X, Z) = \bk[\Hom_{\VI}(\aF^k, Z)] \otimes_{\bk[\GL_k]} D^d_k(X, \aF^k)$.
		\item Given a $\VI$-morphism $\ell \colon X \to Y$, the natural map \[\ell_{\star} \colon D^d_k(X, Z)  \to D^d_k(Y, Z) \]  given by $f \mapsto \tau^{Z}(\ell)f$ is a split injection of $\VI$-modules in the variable $Z$.
	\end{enumerate}
\end{lemma}
\begin{proof}
	The first two parts are immediate from the previous lemma.  Since $\ell \colon X \to Y$ is an injection, it admits an $\aF$-linear section $s \colon Y \to X$ (which may not be an injection). This defines a map $\psi \colon D^d_k(Y, Z) \to D^d_k(X, Z) $ given by \[\ f \mapsto \begin{cases}
	\tau^{Z}(s) f &\mbox{if } \tau^{Z}(s) f \mbox{ is injective}\\
	0 &\mbox{if } \tau^{Z}(s) f \mbox{ is not injective.}
	\end{cases} \]  This map is clearly functorial in $Z$ and is a section to $\ell_{\star}$, finishing the proof.
\end{proof}

The functor $\tau^X$ induces an exact functor $\Sigma^X$, which we again call the {\bf shift functor}, on $\ModVI$ given by $(\Sigma^X M) (Y) = M(\tau^X (Y)) = M(X + Y)$. An element $\phi \in \Aut(Y)$ acts on $(\Sigma^X M)(Y) = M(X+Y)$ where the action is induced by $\tau^X(\phi)$. Similarly, there is an action of $\Aut(X)$ on $\Sigma^X M(Y)$. 

\begin{proposition}
	\label{prop:shift-tensor}
	We have the following: \begin{enumerate}
		\item $\Sigma^X \cI(d) = \bigoplus_{0 \le k \le d} \cI(k) \otimes_{\bk[\GL_k]}D^d_k(X, \aF^k)$. 
		
		\item $ \Sigma^X \cI(W)=\Sigma^X \cI(d)\otimes_{\bk[\GL_d]}W= \bigoplus_{0 \le k \le d } \cI(k) \otimes_{\bk[\GL_k]}D^d_k(X, \aF^k) \otimes_{\bk[\GL_d]} W$.
	\end{enumerate} where  $W$ is any $\bk[\GL_d]$-module. In particular, shift of an induced module is induced, and shift of a projective $\VI$-module is projective.
\end{proposition}
\begin{proof} Since every $\VI$-morphism $f \colon \aF^d \to X + Z$ is of $X$-rank $k$ at most $d$, we  have an isomorphism $\Sigma^X \cI(d)(Z)  = \bigoplus_{0 \le k \le d } D^d_k(X, Z)$. This isomorphism is clearly functorial in $Z$. The rest follows from the previous lemma.
\end{proof}

\begin{corollary}
	\label{cor:shift-of-semi-induced-and-presentation}
	The shift of an induced (semi-induced) $\cC$-module is induced (respectively semi-induced). The category of modules generated (presented) in finite degrees is stable under shift. In particular, $t_0(\Sigma^X M) \le t_0(M)$.
\end{corollary} 
\begin{proof} Exactness of the shift functor  and the previous proposition yields the first assertion. The second assertion follows from Proposition~\ref{prop:generation-presentation} and the previous proposition.
\end{proof}

Suppose $\ell \in\Hom_{\VI}(X, Y)$, and $\tau^{\ell} \colon \tau^X \to \tau^Y$ be the corresponding natural transformation. If $M$ is a $\VI$-module, then $\tau^\ell$ naturally induces a map  $\Sigma^{\ell} \colon \Sigma^X M \to \Sigma^Y M$ which is functorial in $M$. We denote the cokernel of this map by $\Delta^{\ell} M$. When $X = 0$, we simply denote this cokernel by $\Delta^Y$, or simply $\Delta$ if we also have $\dim_{\aF} Y =1$ .  


\begin{proposition}
	\label{prop:split-injection} Let $W$ be a $\VB$-module. Then  $\Sigma^{\ell} \colon \Sigma^X \cI(W) \to \Sigma^Y \cI(W)$ is split injective  and $\Delta^{\ell} \cI(W)$ is an induced module. 
\end{proposition}
\begin{proof}
If $f \colon \aF^d \to X+ Z$ is of $X$-rank $k$ then $\tau^Z(\ell) f$ is clearly of $Y$-rank $k$. Thus $\ell_{\star}$ takes the $k$th direct summand of $\Sigma^X \cI(d)(Z)  = \bigoplus_{0 \le k \le d } D^d_k(X, Z)$ to the $k$th direct summand of $\Sigma^Y \cI(d)(Z)  = \bigoplus_{0 \le k \le d } D^d_k(Y, Z)$, and is functorial in $Z$. Thus it suffices to show that the map $\ell_{\star} \colon D^d_k(X, Z) \to D^d_k(Y, Z)$ is split and the cokernel is induced. That it is split is proven in Lemma~\ref{lem:rank-decomposition}(c), and that the cokernel is induced follows from Lemma~\ref{lem:rank-decomposition}(b) and Proposition~\ref{prop:induced-from-d}. This proves the result when $W = \bk[\Hom_{\VB}(\aF^d, -)]$. The general result follows by observing that tensoring preserves split injections.
\end{proof}

The following basic result is easy to establish.

\begin{proposition} 
	\label{prop:sigma-injectivity}
	Let $\ell \in \Hom_{\VI}(X,Y)$ and $M$ be a $\VI$-module. Then \begin{enumerate}
		\item The shift commutes with $\Gamma$. In particular, $h_0(\Sigma^X M) = \max(h_0(M) - \dim X, -1)$.
		\item The kernel of $\Sigma^{\ell}  \colon \Sigma^X M \to \Sigma^Y M$ is a torsion module of degree $h_0(\Sigma^X M)$.  In particular, $\Sigma^{\ell}  \colon \Sigma^X M \to \Sigma^Y M$ is injective if $\dim X > h_0(M)$.
	\end{enumerate}
\end{proposition}

\subsection{The shift and the difference functors II} We define another shift-like functor $\newshift$ which has better formal properties than $\Sigma$. We first set some notation. Let $\cF$ be a flag on a vector space $Z$ given by \[ 0 = Z_0 \subset Z_1 \subset \ldots \subset Z_n = Z.\] 	We call the stabilizer of $\cF$ in $\GL(Z)$ the parabolic subgroup corresponding to $\cF$ and denote it by $\bP(\cF)$. The {\bf unipotent radical} of $\bP(\cF)$ is the kernel of the natural map \[ \bP(\cF) \to  \prod_{i=1}^n \GL(Z_{i}/Z_{i-1})   \]   and is denoted  by $\bU(\cF)$. Fix a maximal flag \[ 0 = X_0 \subset X_1 \subset \ldots \subset X_n = X. \] In particular, $n$ is equal to the dimension of $X$. Set $Z_0 =  0$ and $Z_{i+1} = X_i + Z$ for $i\ge 0$. Denote the unipotent radical corresponding to the flag \[ 0 = Z_0 \subset Z_1 \subset \ldots \subset Z_{n+1} = X + Z \] by $\bU_X(Z)$. Then $\bU_X$ given by $Z  \mapsto \bU_X(Z) \cong Z^{\dim X} \rtimes \bU_X(0)$ is clearly a $\VI$-group, that is, $\bU_X$ is a functor from $\VI$ to groups. This is in contrast with $Z \mapsto \GL(Z)$, which does not define a $\VI$-group. We define $\newshift^X$ on $\ModVI$ (or $\Mod_{\VB}$) by $\newshift^X M = (\Sigma^X M)_{\bU_X}$, that is, \[\newshift^X M(Z) = M(X+Z)_{\bU_X(Z)}.  \] It is not hard to see that if $M$ is a $\VI$-module then $\newshift^X M$ is a $\VI$-module. In fact, all we need to check is that for every $\VI$-morphism $f \colon Z \to Z', a \in M(X+Z)$ and $\sigma \in \bU_X(Z)$ there exists a $\sigma' \in \bU_X(Z')$ such that  $\tau^X(f)_{\star} (\sigma_{\star} a - a) = \sigma'_{\star} \tau^X(f)_{\star} a -  \tau^X(f)_{\star} a $. But one can simply take $\sigma'$ to be $f_{\star} \sigma$ (the last expression makes sense because $\bU_X$ is a $\VI$-group) and check that the equation holds. Thus $\newshift^X \colon \ModVI \to \ModVI$ is a functor. Here we have suppressed the choice of flag on $X$. We drop the superscript $X$ from $\Sigma^X$ (or $\newshift^X$) when $X$ is of dimension 1.

Suppose we are given an $\ell \in \Hom_{\VI}(X,Y)$ and maximal flags of $X$ and $Y$ such that $\ell$ takes the flag on $X$ to an initial segment of the flag on $Y$. Any $\sigma \in \bU_Y(Z)$ stabilizes $\ell(X) + Z$ and hence can be identified with an element of $\bU_X(Z)$. This induces a  surjection $\ell^{\star} \colon \bU_Y \to \bU_X$ of $\VI$-groups. If $M$ is a $\VI$-module then we can make $\bU_Y$ act on $\Sigma^X M$ via $\ell^{\star}$. Moreover, the map $\Sigma^{\ell} \colon \Sigma^X M \to \Sigma^Y M$ is $\bU_Y$-equivariant. We define $\newshift^{\ell} = \Sigma^{\ell}_{\bU_Y}$ and $\newdelta^{\ell} = \Delta^{\ell}_{\bU_Y}$.  Clearly, we have $(\Sigma^X M)_{\bU_Y} = \newshift^X M$. So $\newshift^{\ell}$ is a map from $\newshift^X M$ to $\newshift^Y M$. It is not hard to see that $\newshift^{\ell}$ is a map of $\VI$-modules. When $X = 0$, there is a unique map $\ell \in \Hom_{\VI}(X,Y)$, so in this case we drop the notation $\Sigma^{\ell}$ and simply call the map $M \to \newshift^Y M$ the natural map. We now note down some basic properties of $\newshift$ that we will use.

\begin{lemma}
	\label{lem:newshift-injectivity}
	In the non-describing characteristic, if $\Sigma^{\ell}$ is injective then so is $\newshift^{\ell}$. In particular, $\newshift^{\ell}$ is injective if $\dim X > h_0(M)$.
\end{lemma}
\begin{proof} This is clear because the size of the group $\bU_Y(Z)$ is invertible in $\bk$ for each $Z$, and $\Sigma^{\ell}$ is $\bU_Y$-equivariant.
\end{proof} 

The lemma immediately implies the following proposition.

\begin{proposition}
	\label{prop:newshift-injectivity}
	Let $\ell$ be the unique map from $0$ to $X$. In the non-describing characteristic, the kernel of the map $\newshift^{\ell} \colon M \to \newshift^X M$ is torsion. In particular, if $M$ is torsion-free then $ \newshift^{\ell}$ is injective.
\end{proposition}

\begin{proposition}
	\label{prop:decreasing-function}
	$t_0(\newshift^X M) \le t_0(\Sigma^X M) \le  t_0(M)$.
\end{proposition}
\begin{proof}
	The first inequality follows from the surjection $\Sigma^X M \to \newshift^X M$. The second is proven in Corollary~\ref{cor:shift-of-semi-induced-and-presentation}.
\end{proof}

\begin{remark} It is not true that
	$\Sigma^X \newshift^Y = \newshift^Y \Sigma^X$. In general, we only have a surjection $\newshift^Y \Sigma^X M \to \Sigma^X \newshift^Y M$. Since we have suppressed the data of the flag on $X + Y$ from $\newshift^{X + Y}$, we will be careful to never interchange $X$ and $Y$. We adopt the convention that an initial segment of the maximal flag on $X+Y$ forms an initial segment of a maximal flag on $Y$ (and not $X$). 
\end{remark}

\begin{proposition} 
\label{prop:iterated-shift}	
We have the following natural isomorphisms:
	\begin{enumerate}
		\item $\Sigma^{X + Y} = \Sigma^Y \Sigma^X$
		\item $\newshift^{X+Y} = \newshift^Y \newshift^X$.
	\end{enumerate} In particular, $\Sigma^X$ is isomorphic to $(\dim X)$-fold iterate of $\Sigma$. The same holds for $\newshift^X$.
\end{proposition} 
\begin{proof}
Part (a) is trivial. Note that we have a short exact sequence of $\VI$-groups \[ 1 \to \Sigma^Y\bU_X \to \bU_{X+Y} \to i_X (\bU_Y) \to 1 \] where $i_X(Z) \colon \GL(Y + Z) \to \GL(X + Y+Z)$ is the natural map. Part (b) now follows from \[ \newshift^{X+Y} M =  (\Sigma^{X+Y}M)_{\bU_{X+Y}} = ((\Sigma^{X+Y} M)_{\Sigma^Y \bU_X})_{i_X(\bU_Y)} = (\Sigma^Y \newshift^X M)_{i_X(U_Y)} = \newshift^Y \newshift^X M.  \qedhere \] 
\end{proof}

The following proposition is the most crucial for our purpose.

\begin{proposition}
	\label{prop:categorical-derivation}
	Let $X$ be a vector space of dimension one. Then $\newshift^X$ is a categorical derivation, that is,  we have \[\newshift^X (M \otimes N)  = (\newshift^X M \otimes N) \bigoplus (M \otimes \newshift^X N).\] In particular, $\newshift \cI(V) = \cI(V) \oplus \cI(\newshift V)$ and  $\newdelta \cI(V) = \cI(\newshift V)$.
\end{proposition}
\begin{proof} 
	Let $V \le W + X$. Then either $V$ is contained in $W$ and $\bU_X(W)$ acts trivially on $V$, or there is an element $\sigma \in \bU_X(W)$ such that $\sigma V$ is of the form $V' + X$ for some subspace $V'$ of $W$. Moreover, if $\tau V = V'' +X$ for some $V'' \le W$ and $\tau \in \bU_X(W)$ then we must have $V' + X = V'' +X$. This shows that $\sigma^{-1} \tau \in \bU_X(V')$. Thus we have \begin{align*}
	\newshift^X (M\otimes N)(W) &= (M \otimes N)(W + X)_{\bU_X(W)} \\
	&= \left(\bigoplus_{V \le W + X}   M((W+X)/V) \otimes_{\bk} N(V) \right)_{\bU_X(W)}\\
	&=  \left( \bigoplus_{V \le W}   M((W+X)/V) \otimes_{\bk} N(V) \right)_{\bU_X(W)} \\
	&\bigoplus \left( \bigoplus_{V' \le W}   M((W+X)/(V'+X)) \otimes_{\bk} N(V'+X)   \right)_{\bU_X(V')}\\
	&=  \left( \bigoplus_{V \le W}   M(W/V+X)_{\bU_X(W/V)} \otimes_{\bk} N(V) \right) \\
	&\bigoplus \left( \bigoplus_{V' \le W}   M(W/V') \otimes_{\bk} N(V'+X)_{\bU_X(V')}  \right)\\
	&= (\newshift^X M \otimes N)(W) \bigoplus (M \otimes \newshift^X N)(W).
	\end{align*} This completes the proof of the first assertion. For the second assertion, just note that $\cI(V) = \bA \otimes V$ and apply the previous part.
 \end{proof}

We have the following basic observations.

\begin{lemma}
	\label{lem:isomorphism-at-projectives}
	Let $A,B \colon \cC_1 \to \cC_2$ be two functors between Grothendieck categories.  Suppose there is a natural transformation $\Psi \colon A \to B$ such that $\Psi(P)$ is an isomorphism for each projective object $P \in \cC_1$. If $A,B$ are right exact then $\Psi(M)$ is an isomorphism for each $M \in \cC_1$.
\end{lemma}

\begin{lemma}
	\label{lem:factors-at-projectives}
	Let $A,B,C$ be right exact functors between two Grothendieck categories $\cC_1, \cC_2$. Suppose there are natural transformations \[A \xrightarrow{\Psi} B \xrightarrow{\Phi} C\] such that for each projective $P \in \cC_1$,  the composite $A(P) \to B(P) \to C(P)$ vanishes. Then $\Phi$ factors through $\coker(\Psi)$.
\end{lemma}

The part (b) of the proposition below is motivated by the footnote in \cite{church-bounding}.

\begin{proposition} 
\label{prop:commuting-pairs}	Let $X$ and $Y$ be vector spaces of dimension one. We have the following equality of functors: \begin{enumerate}
		\item $\Sigma^X \Delta^Y = \Delta^Y \Sigma^X.$
		\item $\rH^{\VI}_0 \newdelta = \newshift \rH^{\VI}_0.$
	\end{enumerate}
\end{proposition}
\begin{proof} 
	The proof of part (a) is identical to \cite[Proposition~1.4 (5)]{djament-vespa}.  We provide a proof sketch here. In the following natural commutative diagram, \begin{figure}[h]
		\[\begin{tikzcd}
		\Sigma^Y \ar{r} \ar{d} & \Sigma^Y \Sigma^X  \ar{d} \\
		\Sigma^Y \ar{r} & \Sigma^X \Sigma^Y
		\end{tikzcd}\]
	\end{figure} the vertical arrows are isomorphisms and so the cokernel of the horizontal maps are also isomorphic. This shows that \[\Sigma^Y \Delta^X = \coker(\Sigma^Y \to \Sigma^Y \Sigma^X) = \coker(\Sigma^Y \to \Sigma^X \Sigma^Y) = \Delta^X \Sigma^Y,\] completing the proof of (a).

		Proof of (b). Composing the natural transformation $\id \to \rH^{\VI}_0$ with $\newshift$ we obtain $\newshift  \to \newshift \rH^{\VI}_0$. Since $\rH^{\VI}_0 \newshift \rH^{\VI}_0 = \newshift \rH^{\VI}_0 $, we obtain a transformation $\rH^{\VI}_0 \newshift  \to \newshift \rH^{\VI}_0$. We shall now apply Lemma~\ref{lem:factors-at-projectives} to the composite \[ \rH^{\VI}_0 \to \rH^{\VI}_0 \newshift  \to \newshift \rH^{\VI}_0. \] To check the hypothesis of the lemma, it is enough to assume that $P =\cI(V)$ where $V$ is concentrated in degree $d$ (Proposition~\ref{prop:projectives}). Evaluating the composite  above at $P$ yields \[V \to V \oplus \newshift V \to \newshift V. \] From degree considerations, hypothesis of Lemma~\ref{lem:factors-at-projectives} is satisfied. Thus we conclude that there is a natural transformation $\rH^{\VI}_0 \newdelta  \to \newshift \rH^{\VI}_0$. By Lemma~\ref{lem:isomorphism-at-projectives} and Proposition~\ref{prop:projectives}, this transformation is an isomorphism. This completes the proof.
\end{proof}

\begin{remark} There does not seem to be an equivalence between $\newshift^X \newdelta^Y$ and  $\newdelta^Y \newshift^X$. This is in contrast with the case of $\FI$-modules.
	\end{remark}

We denote the kernel of the natural transformation $\id \to \newshift^X$ by $\kappa^X$.

\begin{proposition}
	\label{prop:delta-1} In the non-describing characteristic, 	we have $\rL_1 \newdelta^X = \kappa^X$, and $\rL_i \newdelta^X = 0$ for $i > 1$.
\end{proposition}
\begin{proof}
The proof is the same as that of \cite[Lemma~4.7]{castelnuovo-regularity}, where $\newshift^X$ plays the role of $S$. We provide a proof sketch here. Given a $\VI$-module $M$, we can find a presentation \[ 0 \to K \to F \to M \to 0, \] where $F$ is a projective $\VI$-module, and $K$ is torsion-free. The corresponding long exact sequence for the right exact functor $\newdelta^X$ implies that $\rL_1 \newdelta^X(M) = \ker(\newdelta^X K \to \newdelta^X F)$. Note that $F \to \Sigma^X F$ is injective, as $F$ is torsion-free. By Lemma~\ref{lem:newshift-injectivity}, we conclude that $F \to \newshift^X F$ is injective. Thus we have the following commutative diagram \begin{figure}[h]
	\[\begin{tikzcd}
& K \ar{r} \ar{d} & \newshift^X K \ar{r}   \ar{d} & \newdelta^X K \ar{r} \ar{d} & 0 \\
0  \ar{r}  & 	F \ar{r} & \newshift^X F \ar{r}    & \newdelta^X F \ar{r}  & 0.
	\end{tikzcd}\]
\end{figure} Applying the snake lemma, we see that \[ \ker(\newshift^X K \to \newshift^X F) =0 \to \rL_1 \newdelta^X(M) \to M \to \newshift^X M \to \newdelta^X M \to 0.  \] This shows that $\rL_1 \newdelta^X(M) = \kappa^X(M)$, finishing the proof of the first assertion.  By dimension shifting, we have $\rL_2 \newdelta^X(M) = \rL_1 \newdelta^X(K) = \kappa^X(K)$. Since $K$ is torsion-free, we see that $\rL_2 \newdelta^X(M) = 0$. Since $M$ is arbitrary it follows that $\rL_i \newdelta^X = 0$ for $i > 1$.
\end{proof}

The following lemma is proven in a similar way as \cite[Proposition~1.4 (7)]{djament-vespa}.

\begin{lemma}
\label{lem:djament-composable-morphisms}
Let $M$ be a $\VI$-module, and $X,Y$ be vector spaces. We have an exact sequence of the form: \[\newdelta^Y M \to \newdelta^{X + Y} M \to \newshift^Y \newdelta^X  M \to 0  .\] Moreover, in the non-describing characteristic, this can be extended to \[ 0 \to \kappa^Y M \to \kappa^{X + Y} M \to \newshift^Y \kappa^X M \to \newdelta^Y M \to \newdelta^{X + Y} M \to \newshift^Y \newdelta^X  M \to 0  .\]
\end{lemma}
\begin{proof} Let $\ell \colon 0 \to Y$, $\ell' \colon 0 \to X$ and $\ell'' \colon 0 \to X + Y$ be natural maps. Then we have composable maps $\newshift^{\ell} \colon M \to \newshift^Y M$ and $\newshift^Y \newshift^{\ell'} \colon \newshift^Y M \to \newshift^Y \newshift^X M$, where the composite is $(\newshift^Y \newshift^{\ell'}) \circ \newshift^{\ell} = \newshift^{\ell''}$. Two composable morphisms $u,v$ in an abelian category induce an exact sequence (\cite[Exercise 6, \S II.5]{maclane}) \[0 \to \ker(u) \to \ker(v \circ u) \to \ker(v) \to \coker(u)   \to \coker(v \circ u) \to \coker(v) \to 0.\] Set $u = \newshift^{\ell}$  and $v = \newshift^Y \newshift^{\ell'} $. Since $\newshift^Y$ is right exact we see that $\coker v = \newshift^Y \newdelta^X M$ and the first assertion follows. In non-describing characteristic, $\newshift^Y$ is exact. Thus we have $\ker(v) = \newshift^Y \kappa^X M$. This finishes the proof.
\end{proof}

\begin{corollary} 
	\label{cor:generation-of-newdelta}
Let $X$ and $Y$ be vector spaces, and fix maximal flags on $X$ and $Y$. Let $\ell \in \Hom_{\VI}(X,Y)$ be a map that takes the maximal flag on $X$ to an initial segment of the flag on $Y$. Then $t_0(\newdelta^{\ell} M) < t_0(M)$.
\end{corollary}
\begin{proof} Choose a complement $Z$ of $\ell(X)$ in $Y$. Then the maximal flag on $Y$ will induce a maximal flag on $Z$. We can identify $\ell$ with $\tau^X(\ell')$ where $\ell' \colon 0 \to Z$. This shows that $\newshift^{\ell} = \newshift^X \newdelta^{\ell'} = \newshift^X \newdelta^Z$. Thus by Corollary~\ref{cor:shift-of-semi-induced-and-presentation}, it is enough to show that  $t_0(\newdelta^{Z}M) < t_0(M)$. By the previous lemma, it suffices to prove it in the case when $\dim Z = 1$. But in this case, we have $t_0(\newdelta^Z M) = \deg (\newshift \rH^{\VI}_0(M)) < t_0(M)$ (see Proposition~\ref{prop:commuting-pairs}). This completes the proof.
\end{proof}

 \subsection{Derived saturated objects}
 Our aim here is to show that the semi-induced modules are always derived saturated, and that the converse holds in the non-describing characteristic. We recall that a module $M$ is derived saturated if and only if $\rR \Gamma(M) = 0$ (Proposition~\ref{prop:triangle}). 
 
%
%
 \begin{lemma}
 	\label{lem:Sigma-preserves-Gamma-acyclics}
 	The natural map $\Sigma (\rR^i \Gamma)(M) \to (\rR^i \Gamma) \Sigma M$ is an isomorphism. Equivalently, $\Sigma$ preserves $\Gamma$-acyclic objects. 
 \end{lemma}
\begin{proof}
	We follow the argument in \cite[Proposition~A.3]{djament} to prove our assertion. The proof is by induction on $i$. The base case $i=0$ is immediate as $\Sigma$ commutes with $\Gamma$. Suppose that $i > 0$, and that the result has been proven for $j < i$. 
	
	We first   apply a dimension shifting argument to see that the natural map $\Sigma (\rR^i \Gamma)(M) \to (\rR^i \Gamma) \Sigma M$ is injective. To see this, consider any exact sequence \[ 0 \to M \to I \to N \to 0 \] where $I$ is an injective. This yields a commutative diagram	\[
	\begin{tikzcd}
	\Sigma (\rR^{i-1} \Gamma) I \ar{r} \ar{d} & \Sigma (\rR^{i-1} \Gamma)(N) \ar{r} \ar{d} & \Sigma (\rR^i \Gamma)(M) \ar{r} \ar{d} & 0 \ar{d} \\ 
	(\rR^{i-1} \Gamma) \Sigma I \ar{r} & (\rR^{i-1} \Gamma) \Sigma N \ar{r}  & (\rR^i \Gamma) \Sigma M \ar{r}  & (\rR^i \Gamma) \Sigma I
	\end{tikzcd}
	\] whose rows are exact. By induction, the first two vertical arrows are isomorphisms. Thus by the four Lemma, we see that the third vertical arrow is injective. 
	
 By Lemma~\ref{lem:torsion-injective}, we see that $\rR^k \Gamma N =0$ whenever $k >0$ and $N$ is a torsion module. Thus $\rR^i \Gamma (M/\Gamma(M)) = \rR^i \Gamma M$ for any $i>0$.  Given a vector space $X$, we have the following natural exact sequence \[ 0 \to M/\Gamma(M) \to \Sigma^X M \to \Delta^X M \to 0.\] By the corresponding long exact sequence for $\Gamma$, we obtain the following exact sequence \[ \rR^{i-1} \Gamma (\Sigma^X M) \to \rR^{i-1} \Gamma( \Delta^X M) \to \rR^{i} \Gamma ( M) \to \Sigma^X  \rR^{i} \Gamma (M),    \] where the exactness comes from the injectivity of the map $\Sigma^X (\rR^i \Gamma)(M) \to (\rR^i \Gamma) \Sigma^X M$ proved in the previous paragraph. We conclude that \[\ker( \rR^{i} \Gamma(M)  \to (\rR^i \Gamma) \Sigma^X M) = \coker (\rR^{i-1} \Gamma (\Sigma^X M) \to \rR^{i-1} \Gamma( \Delta^X M)).\] Since $\Sigma$ is exact, and commutes with $\Sigma^X$ and $\Delta^X$ (Proposition~\ref{prop:commuting-pairs}), we see that \begin{align*}
 \Sigma \ker( \rR^{i} \Gamma(M)  \to (\rR^i \Gamma) \Sigma^X M)  &=\Sigma  \coker (\rR^{i-1} \Gamma (\Sigma^X M) \to \rR^{i-1} \Gamma( \Delta^X M))\\
 &=  \coker (\rR^{i-1} \Gamma (\Sigma^X \Sigma M) \to \rR^{i-1} \Gamma( \Delta^X \Sigma M)) &\mbox{(by induction)}\\
&= \ker( \rR^{i} \Gamma(\Sigma M)  \to (\rR^i \Gamma) \Sigma^X \Sigma M)
 \end{align*} Thus $\Sigma$ commutes with $\ker(\id \to \Sigma^X) \circ (\rR^i \Gamma)$ for any $X$. Since $X$ is arbitrary and $\Sigma$ is cocontinuous, we see that $\Sigma$ commutes with $\rR^i \Gamma$. This finishes the proof.
\end{proof}

 The following result is motivated by \cite[Proposition~1.1]{djament}.
 
%

 \begin{proposition}
 	\label{prop:induced-is-torsion-acyclic}
 	If $F$ is an induced $\VI$-module, then $\rR^i \Gamma(F) = 0$ for $i \ge 0$.
 \end{proposition}
 \begin{proof}
%
%
%

We have the following natural commutative diagram: 	\[
\begin{tikzcd}
\rR^i\Gamma (F) \ar{rr}{\Sigma^{\ell}} \ar{rrd}[swap]{\rR^i\Gamma (\Sigma^{\ell})} & {} & \Sigma^X \rR^i\Gamma (F) \ar{d} \\ 
{} & {} &  \rR^i\Gamma (\Sigma^X F)  
\end{tikzcd}
\]where $\ell$ is the map from $0$ to $X$. Since  $\Sigma^{\ell}$, applied to $F$, is split-injective (Proposition~\ref{prop:split-injection}), we see that $\rR^i\Gamma (\Sigma^{\ell})$ is injective. By the previous lemma, the vertical map is an isomorphism. Thus the map $\Sigma^{\ell} \colon \rR^i\Gamma(F) \to \Sigma^X \rR^i\Gamma(F)$ is injective as well.  Since $X$ is arbitrary, we see that $\rR^i\Gamma(F)$ is torsion-free. By definition, $\rR^i\Gamma(F)$ is also a torsion $\VI$-module. Hence $\rR^i\Gamma(F) = 0$.
 \end{proof}
 
 \begin{corollary}
 	\label{cor:semi-induced-are-derived-saturated}
 	Semi-induced modules are derived saturated.
 \end{corollary}

\begin{corollary}
	\label{cor:short-exact-sequence-semi-induced}
	In a short exact sequence, if two of the objects are semi-induced then so is the third.
\end{corollary}

\begin{proof} {\bf Note: There is a mistake in this published version of the proof which was pointed out to us by Powell. See the Errata below.}
	Let $0 \to L \to M \to N \to 0$ be an exact sequence of modules presented in finite degrees. Then there exists a $d$ such that $L, M,  N$ are generated in degree $\le d$. We proceed by induction on $d$. First suppose that $N$ is semi-induced. In this case, $\rH^{\VI}_1(L) =0$ if and only if $\rH^{\VI}_1(M) = 0$. So the result follows from Proposition~\ref{prop:ramos-new}. Now suppose that $L$ and $M$ are semi-induced. By the previous corollary, $N$ is derived-saturated. In particular, $N$ is torsion-free. We claim that \[0 \to L_{\prec d} \to M_{\prec d} \to N_{\prec d} \to 0\] is an exact sequence. To see this, first note that we have a natural exact sequence \[0 \to L\cap M_{\prec d} \to M_{\prec d} \to N_{\prec d} \to 0, \] and that $ L_{\prec d} \subset L\cap M_{\prec d}$. Now suppose, if possible,  $x$ is in $L\cap M_{\prec d}$ but not in $L_{\prec d}$. Then there exists a $y \in M(\aF^{d-1})$ and a $\VI$-morphism $f$ such that $f_{\star}(y) = x \in L$. Since $x$ is not in $L_{\prec d}$, we see that $y \notin L(\aF^{d-1})$. Let $\ol{x}, \ol{y}$ be the images of $x$ and $y$ in $N$. Then $\ol{y} \neq 0$, but $ f_{\star}(\ol{y}) =\ol{x} = 0$. This contradicts the fact that $N$ is torsion-free, proving the claim.

	By induction, $N_{\prec d}$ is semi-induced. Thus it suffices to show that $N/N_{\prec d}$ is induced from $d$. By applying the snake lemma to the diagram below, 
		\[\begin{tikzcd}
		0 \ar{r}  & L \ar{r} & M \ar{r} & N \ar{r} & 0 \\ 
		0 \ar{r}  & L_{\prec d} \ar{u} \ar{r} & M_{\prec d} \ar{u}  \ar{r} & N_{\prec d} \ar{u}  \ar{r} & 0
		\end{tikzcd}\]
we obtain an exact sequence \[0 \to L_{\preceq d}/L_{\prec d} \to M_{\preceq d}/M_{\prec d} \to N_{\preceq d}/N_{\prec d} \to 0.\] Since the first two objects in this exact sequence are induced from $d$, so is the third (Proposition~\ref{prop:induced-from-d}). This completes the proof.
\end{proof} 

\begin{proof}[{\bf Errata (corrected proof)}]
	The part of the argument in the proof above which shows that ``if $L$ and $M$ are semi-induced then so is $N$'' is incorrect as pointed out to us by Powell. The error above is that $x$ should be a sum of element of the form $f_{\star}(y)$, instead of being equal to one such element. Here we provide a corrected argument due to Powell for that part:
	
Suppose that the injection $L \to M$ is given by $\phi$, and that $L$ and $M$ are generated in degrees $\le d$.	We show by induction on $d$ that $\coker(\phi)$ is semi-induced. If $d=0$, then $L$ and $M$ are induced, and so the claim follows from Proposition~\ref{prop:induced-from-d}. Assume now that $d >0$. The map $\phi$ induces a map $\psi \colon L/L_{\prec d} \to M/M_{\prec d}$ of modules induced from $d$. Thus by Proposition~\ref{prop:induced-from-d}, the kernel of $\psi$ is induced from $d$.

Let $L'$ be the pullback of $\ker(\psi)$ to $L$. Then we have an exact sequence \[ 0 \to L_{\prec d} \to L' \to \ker(\psi) \to 0,  \] and so $L'$ is semi-induced (by the correct part of the argument for the corollary above). Clearly, $\phi$ restricts to an injection $L' \to M_{\prec d}$. Reducing modulo $L_{\prec d}$, we get an injection $\ker(\psi) \to M_{\prec d}/L_{\prec d}$. Since $M_{\prec d}/L_{\prec d}$ generated in degrees $< d$, the lemma below shows that $\ker(\psi) =0$. By induction, the cokernel of the restriction $\phi_{\prec d} \colon L_{\prec d} \to M_{\prec d}$ of $\phi$ is semi-induced.  Now by the snake lemma applied to the diagram 	\[\begin{tikzcd}
0 \ar{r}  & L \ar{r} & M \ar{r} & \coker(\phi) \ar{r} & 0 \\ 
0 \ar{r}  & L_{\prec d} \ar{u} \ar{r} & M_{\prec d} \ar{u}  \ar{r} & \coker(\phi_{\prec d}) \ar{u}  \ar{r} & 0
\end{tikzcd}\]
we obtain an exact sequence \[0 \to L/L_{\prec d} \xrightarrow{\psi} M/M_{\prec d} \to \coker(\phi)/\coker(\phi_{\prec d}) \to 0.\] By Proposition~\ref{prop:induced-from-d}, we see that $\coker(\phi)/\coker(\phi_{\prec d})$ is induced from $d$. Since $\coker(\phi_{\prec d})$ and $\coker(\phi)/\coker(\phi_{\prec d})$ are semi-induced, we see that $\coker(\phi) =N$ is semi-induced as well (by the correct part of the argument for the corollary above). 
\end{proof}	

\begin{lem}[This lemma does not appear in the published version] Let $W$ be a nonzero $\bk[\GL_d]$-module, and let $M$ be a $\VI$-module generated in degrees $< d$. Then there are no injective maps from $\cI(W)$ to $M$.
\end{lem}
\begin{proof}
Without loss of generality, we may assume that $W$ is singly generated by $w \in W$ and $M$ is finitely generated. We prove the assertion by noetherian induction on $\bk$.

Suppose first that $\bk$ is a field. Then $\dim_{\bk} \cI(W)(\aF^n) $ is a polynomial in $q^n$ of degree exactly $d$  (see Lemma~\ref{lem:dimension} as well). On the other hand, $M$ is a quotient of an induced module of the form $\cI(V)$ where $V$ is finitely generated and concentrated in degrees $<d$. And so $\lim_{n \to \infty} \frac{\dim_{\bk} M(\aF^n)}{q^{nd}}  = 0$. Thus the assertion holds over fields.

Next suppose that $\bk$ is an integral domain, and let $K$ be its fraction field. Suppose, if possible, there is an injection $\phi \colon \cI(W) \to M$.  Since $K$ is flat over $\bk$ we see that $K \otimes_{\bk}  \cI(W) \to K \otimes_{\bk}  M$ is an injection. By the noetherian induction (the  field case above), we see that  $K \otimes_{\bk}  \cI(W) =0$. And so every element of $W$ is annihilated by a nonzero element of $\bk$. Let $\fa$ be the annihilator of $w \in W$. Note that $\fa$ in nonzero and annihilates $W$.  We have an injection $\Hom_{\bk}(\bk/\fa, \cI(W)) \to \Hom_{\bk}(\bk/\fa, M)$. Note here that $\Hom_{\bk}(\bk/\fa, \cI(W))  = \cI(\Hom_{\bk}(\bk/\fa, W) )$. By noetherian induction, the assertion holds over $\bk/ \fa$. This implies that $\Hom_{\bk}(\bk/\fa, W) = 0$, and so $W$ in not annihilated by $\fa$. This is a contradiction. Thus the assertion holds over integral domains.

Finally suppose that $\bk$ is not an integral domain. Then there are two nontrivial ideals $\fa, \fb \subset \bk$ such that $\fa \fb =0$. Suppose, if possible, there is an injection $\phi \colon \cI(W) \to M$.  Then $\Hom_{\bk}(\bk/\fa, \cI(W)) \to \Hom_{\bk}(\bk/\fa, M)$ is an injection as well. By noetherian induction, we see that $\Hom_{\bk}(\bk/\fa, W) = 0$. It follows that $w$ is not annihilated by $\fa$. Let $a \in \fa$ be such that $a w \ne 0$, and let $W'$ be the submodule of $W$ generated by $a w$. Since $\fa \fb = 0$, $W'$ is annihilated by $\fb$. We have an injection $\Hom_{\bk}(\bk/\fb, \cI(W')) \to \Hom_{\bk}(\bk/\fb, M)$.  By noetherian induction, we see that $\Hom_{\bk}(\bk/\fb, W') = 0$. It follows that $W'$ is not annihilated by $\fb$, a contradiction. Thus the assertion holds over non-integral domains as well. This completes the proof.
\end{proof}	

\begin{question} 
	\label{conj:semi-induced-intersection}
	Let $A, B, N$ be  semi-induced modules, and assume that $A, B \subset N$. Then is it true that $A \cap B$ is semi-induced?
\end{question}

\subsubsection{The case of non-describing characteristic}
We now assume that we are in the non-describing characteristic and prove the converse of Corollary~\ref{cor:semi-induced-are-derived-saturated}. Along the way, we show that $\newshift$ commutes with $\Gamma$ which, indeed, is a crucial step of our proof.

\begin{lemma}
	\label{lem:maschke}
	Let $V$ be a $\bk[G]$-module, and assume that the size of $G$ is invertible in $\bk$. Let $x$ be an element of $V_G$, and let $\wt{x}$ be a lift of $x$ in $V$. Then \begin{enumerate}
		\item $1/|G|\sum_{\sigma \in G} \sigma \wt{x}$ in another lift of $x$.
		\item $x = 0$ if and only if $\sum_{\sigma \in G} \sigma \wt{x} = 0$.
	\end{enumerate}
\end{lemma}
\begin{proof}
	This is a standard result.
\end{proof}

\begin{lemma}
	\label{lem:combinatorial-identity}
	Let $M$ be a torsion-free $\VI$-module, and let $X$ be a vector space. Then $\newshift^X M$ is torsion-free.
\end{lemma}
\begin{proof}
	We may assume that $X$ is of dimension one (Proposition~\ref{prop:iterated-shift}). Let $Y$ be another vector space of dimension one. It suffices to show that the map $f_{\star} \colon \newshift^X M(Z) \to \newshift^X M(Z + Y)$ induced by the inclusion $f \colon Z \to Z +Y$ is injective for every $Z$. Suppose $f_{\star}(x) = 0$ for some $x$. By the previous lemma, there is a lift $\wt{x} \in \Sigma^X M(Z) = M(X+Z)$ of $x$ which is invariant with respect to $\bU_X(Z)$.  Since $f_{\star}(x) =0$ and $f_{\star}(\wt{x}) \in \Sigma^X M(Z + Y) = M(X+Z+Y)$ is a lift of $f_{\star}(x)$, the previous lemma tells us that \[\sum_{\sigma \in \bU_X(Y + Z)} \sigma f_{\star}(\wt{x}) = 0. \] But $\bU_X(Y + Z) = \bU_X(Y) \times \bU_X(Z)$ and $\wt{x}$ is invariant with respect to  $\bU_X(Z)$, and so we conclude that \[\sum_{\sigma \in \bU_X(Y)} \sigma f_{\star}(\wt{x}) = 0. \]  
	
	Let $W$ be the $\VB$ module given by $\bk[\Hom_{\VB}(X' \oplus Z, -)]$ where $X'$ is a one-dimensional space. Fix an isomorphism $\alpha \colon X' + Z \to X + Z$. Then $[\alpha]$ is a generator of the $\VI$-module $\cI(W)$.	There is a unique map $\psi \colon \cI(W) \to M$ which takes $[\alpha]$ to $\wt{x}$. Let $N$ be the $\VI$-submodule of $\cI(W)$ generated by $\sum_{\sigma \in \bU_X(Y)} \sigma f_{\star}([\alpha])$. Then the equation at the end of the last paragraph shows that $\psi$ factors through the projection $\cI(W) \to \cI(W)/N$. We claim that $\psi =0$. Since $M$ is torsion-free and $\psi$ factors through  $\cI(W)/N$, it suffices to show that $\cI(W)/N$ is a torsion module. Fix an isomorphism $h \colon Y \to X$. Let $S$ be the collection consisting of $q-1$ automorphisms of $X+Y+Z$ that fix $Z$, send $Y$ to $X$ via $h$, and send $X$ to $Y$ via a nonzero multiple of $h^{-1}$. Then the following equation can be easily verified: \[ (\sum_{\tau \in \bU_Y(X)} \tau - \sum_{\tau \in S} \tau)(\sum_{\sigma \in \bU_X(Y)} \sigma f_{\star}([\alpha])) = q f_{\star}([\alpha]). \] Since $q$ is invertible, the above equation shows that $f_{\star}([\alpha]) \in N$. This shows that $\cI(W)/N$ is torsion, and so $\psi =0$. This implies that $x = 0$, completing the proof. 
\end{proof}

\begin{proposition}
	\label{prop:newshift-commutes-with-Gamma}
	$\newshift$ commutes with $\Gamma$.
\end{proposition}
\begin{proof}
	Let $M$ be a $\VI$-module, and $X$ be a vector space of dimension one so that $\newshift = \newshift^X$. Since $\newshift$ is exact and $\Gamma M \subset M$, we see that $\newshift \Gamma M \subset \Gamma \newshift M$. For the reverse inclusion, first note that $M/\Gamma M$ is torsion-free. Thus by the previous lemma and the exactness of $\newshift$, we see that \[ \newshift(M/\Gamma M) = (\newshift M)/ (\newshift\Gamma M) \] is torsion-free, and so the torsion part $\Gamma\newshift M$ of $\newshift M$ is contained in $\newshift\Gamma(M)$, completing the proof.
\end{proof}

We now focus on showing that $\newshift$ preserves $\Gamma$-acyclic objects. We need a couple of lemma.

\begin{lemma}[{\cite[Corollaire~A.4]{djament}}] 
\label{lem:differential-criterion}	
	Let $M$ be a $\VI$-modulen, and let $n$ be a non-negative integer. Then the following are equivalent: \begin{enumerate}
		\item $\rR^k \Gamma (M) = 0$ for $0 \le k \le n$.
		\item For each $0 \le  k \le n$ and vector spaces $X_1, \ldots, X_k$, the $\VI$-module $\Delta^{X_1}\Delta^{X_2}\cdots \Delta^{X_k} M$ is torsion-free.
	\end{enumerate}  
\end{lemma}
\begin{proof}
We prove the assertion by induction on $n$. The base case $n=0$ is trivial. Assume now that $n>0$, and that the assertion holds for smaller values of $n$. 

Suppose first that (b) holds. Then, by induction, $\rR^k \Gamma (M) = 0$ for $0 \le k < n$. In particular, $M$ is torsion free. So for any vector space $X$, we have a short exact sequence: \[0 \to M \to \Sigma^{X} M \to \Delta^{X} M \to 0.  \] By induction, $\rR^k \Gamma(\Delta^{X}M) = 0$ for $0 \le k <n$. Thus the long exact sequence corresponding to the exact sequence above yields that $\rR^n \Gamma(M) \to \rR^n \Gamma(\Sigma^X M)$ is injective. We have the following natural commutative diagram: 	\[
\begin{tikzcd}
\rR^n\Gamma (M) \ar{rr}{\Sigma^{\ell}} \ar{rrd}[swap]{\rR^n\Gamma (\Sigma^{\ell})} & {} & \Sigma^X \rR^n\Gamma (M) \ar{d} \\ 
{} & {} &  \rR^n\Gamma (\Sigma^X M)  
\end{tikzcd}
\]where $\ell$ is the map from $0$ to $X$. Since the vertical map is an isomorphism (Lemma~\ref{lem:Sigma-preserves-Gamma-acyclics}), we conclude that the horizontal map is injective as well. Since this holds for each $X$ and $\rR^n\Gamma (M)$ is a torsion module, we have $\rR^n\Gamma (M) = 0$. Thus (a) holds.

Conversely, suppose that (a) holds. Since $n > 0$, the module $M$ is torsion-free. So for any vector space $X$, we have a short exact sequence: \[0 \to M \to \Sigma^{X} M \to \Delta^{X} M \to 0.  \] The corresponding long exact sequence yields $\rR^{k}\Gamma(\Sigma^{X} M) \cong \rR^k\Gamma(\Delta^X M)$ for $0 \le k <n$. By Lemma~\ref{lem:Sigma-preserves-Gamma-acyclics}, we conclude that $\rR^k\Gamma(\Delta^X M) = 0$ for $0 \le k <n$. Now (b) follows immediately from the induction hypothesis. This completes the proof.
\end{proof}

\begin{lemma}
	\label{lem:new-torsion-free}
	Let $M$ be a $\VI$-module, and let $X, X_1, \ldots, X_k$ be vector spaces. Suppose that the $\VI$-module $\Delta^{X_1}\Delta^{X_2}\cdots \Delta^{X_k} M$ is torsion-free. Then $\Delta^{X_1}\Delta^{X_2}\cdots \Delta^{X_k} \newshift^X M$ is torsion-free.
\end{lemma}
\begin{proof}
	By Proposition~\ref{prop:commuting-pairs}, we see that $\Delta^{X_1}\Delta^{X_2}\cdots \Delta^{X_k} \Sigma^X M = \Sigma^X \Delta^{X_1}\Delta^{X_2}\cdots \Delta^{X_k} M$. Set $N = \Delta^{X_1}\Delta^{X_2}\cdots \Delta^{X_k} M$, and note that \[\Delta^{X_1}\Delta^{X_2}\cdots \Delta^{X_k} \newshift^X M(Z) = \Sigma^X \Delta^{X_1}\Delta^{X_2}\cdots \Delta^{X_k} M(Z)_{\bU_X(\sum_{i=1}^k X_i + Z)} =  N(X+ Z)_{\bU_X(\sum_{i=1}^k X_i + Z)}.\] Set $\bV_{X}(-) = \bU_X(\sum_{i=1}^k X_i + -)$.  We now follow the proof of Lemma~\ref{lem:combinatorial-identity} closely.
	
We may assume without loss of generality that $X$ is of dimension one. Let $Y$ be another vector space of dimension one. It suffices to show that the map $f_{\star} \colon \Sigma^X N(Z)_{\bV_X(Z)} \to \Sigma^X N(Z + Y)_{\bV_X(Z+Y)}$ induced by the inclusion $f \colon Z \to Z +Y$ is injective for every $Z$. Suppose $f_{\star}(x) = 0$ for some $x$. By Lemma~\ref{lem:maschke}, there is a lift $\wt{x} \in \Sigma^X N(Z) = N(X+Z)$ of $x$ which is invariant with respect to $\bV_X(Z)$.  Since $f_{\star}(x) =0$ and $f_{\star}(\wt{x}) \in \Sigma^X N(Z + Y) = N(X+Z+Y)$ is a lift of $f_{\star}(x)$, Lemma~\ref{lem:maschke} tells us that \[\sum_{\sigma \in \bV_X(Y + Z)} \sigma f_{\star}(\wt{x}) = 0. \] But $\bV_X(Y + Z) = \bU_X(Y) \times \bV_X(Z)$ and $\wt{x}$ is invariant with respect to  $\bV_X(Z)$, and so we conclude that \[\sum_{\sigma \in \bU_X(Y)} \sigma f_{\star}(\wt{x}) = 0. \]  
	
	Let $W$ be the $\VB$ module given by $\bk[\Hom_{\VB}(X' \oplus Z, -)]$ where $X'$ is a one-dimensional space. Fix an isomorphism $\alpha \colon X' + Z \to X + Z$. Then $[\alpha]$ is a generator of the $\VI$-module $\cI(W)$.	There is a unique map $\psi \colon \cI(W) \to N$ which takes $[\alpha]$ to $\wt{x}$. Let $N'$ be the $\VI$-submodule of $\cI(W)$ generated by $\sum_{\sigma \in \bU_X(Y)} \sigma f_{\star}([\alpha])$. Then the equation at the end of the last paragraph shows that $\psi$ factors through the projection $\cI(W) \to \cI(W)/N'$. We claim that $\psi =0$. Since $M$ is torsion-free and $\psi$ factors through  $\cI(W)/N'$, it suffices to show that $\cI(W)/N'$ is a torsion module. This has already been established in the proof of Lemma~\ref{lem:combinatorial-identity}. So $\psi = 0$. This implies that $x=0$, completing the proof.
\end{proof}

\begin{proposition}
	\label{prop:newshift-preserves-gamma-acyclics}
	The functor $\newshift$ preserves $\Gamma$-acyclic objects.
\end{proposition}
\begin{proof}
	Let $M$ be a $\Gamma$-acyclic object. By Proposition~\ref{prop:triangle} and Lemma~\ref{lem:torsion-injective}, the $\VI$-module $M/\Gamma(M)$ is derived saturated. By Lemma~\ref{lem:differential-criterion}, for each $k \ge 0$ and vector spaces $X_1, \ldots, X_k$, the $\VI$-module $\Delta^{X_1}\Delta^{X_2}\cdots \Delta^{X_k} (M/\Gamma(M))$ is torsion-free. By the previous lemma, for each $k \ge 0$ and vector spaces $X_1, \ldots, X_k$, the $\VI$-module $\Delta^{X_1}\Delta^{X_2}\cdots \Delta^{X_k} \newshift (M/\Gamma(M))$ is torsion-free. By Lemma~\ref{lem:differential-criterion} again, $\newshift (M/\Gamma(M))$ is derived saturated. Since $\Gamma$ commutes with $\newshift$ (Proposition~\ref{prop:newshift-commutes-with-Gamma}), we see that $\newshift M/\Gamma(\newshift M)$ is derived saturated. By Lemma~\ref{lem:torsion-injective}, $\newshift M$ is $\Gamma$-acyclic, completing the proof.
\end{proof}

The following question is quite natural: 

\begin{question} 
	\label{ques:adjoint}
	Do either of $\Sigma$ or $\newshift$ preserve injective objects?  Note: A positive answer is known in the $q=1$ ($\FI$-modules) case; see \cite{negative-one}.
\end{question}

\begin{lemma}
	\label{lem:new-saturated}
If $M$ is derived saturated, then so are $\newshift M$ and $\newdelta M$.
\end{lemma}
 \begin{proof}
Since $\newshift$ commutes with $\Gamma$ (Proposition~\ref{prop:newshift-commutes-with-Gamma})  and preserves $\Gamma$-acyclic objects (Proposition~\ref{prop:newshift-preserves-gamma-acyclics}), we have $\rR \Gamma \newshift M = \newshift \rR \Gamma M = 0$. Thus by Proposition~\ref{prop:triangle}, we see that $\newshift M$ is derived saturated. The result about $\newdelta M$ follows from the exact sequence (see Proposition~\ref{prop:newshift-injectivity}) \[ 0 \to M \to \newshift M \to \newdelta M \to 0. \qedhere \] 
 \end{proof}

\begin{lemma}[Non-vanishing coinvariants]
	\label{lem:non-vanishing-coinvariants}
	Suppose $K \le H \le G$ are finite groups. Let $W$ be a $\bk[H]$-module such that $K$ acts trivially on $W$. Then for any $\bk[G]$-submodule $V$ of $\Ind_{H}^G W$, we have $V_K =0 \Longleftrightarrow V=0$.
\end{lemma}
\begin{proof}
	Let $\cT = \{\tau_1, \ldots, \tau_n\}$ be a full set of representatives in $G$ of the left coset space $G/H$. We assume that $\tau_1 = 1_{G}$. Any element $x \in \Ind_{H}^G W$ can be thought of as a function $x \colon \cT \to W$, and the action of $\sigma \in G$ on $x$ is given by $(\sigma x)(\tau_{n_i}) = h_i x(\tau_i)$ where $h_i \in H$ and $n_i$ are uniquely determined by the equation $\sigma \tau_i = \tau_{n_i} h_i$. As a special case, we note that if $\sigma \in K$, then we have $\sigma = \sigma \tau_1 = \tau_1 h_1 = h_1 \in K$. Since $K$ acts trivially on $W$, we conclude that $(\sigma x)(\tau_1) = \sigma (x(\tau_1)) = x(\tau_1)$.

	Assume now that $V$ is nontrivial.   Let $x \in V$ be a nonzero element. As in the previous paragraph, we think of $x$ as a function from $\cT$ to $W$.  Since $G$ acts transitively on $G/H$, there exists a $\sigma \in G$ such that $\sigma x$ is nonzero on $\tau_1$. Now suppose, if possible, the image of $\sigma x$ in $V_K$ is $0$. Then $\sigma x$ can be written as \[\sigma x = \sum_j (x_j - \sigma_j x_j)\] where $x_j$ are in $\Ind_{H}^G W$, and $\sigma_j$ are in $K$.  By the previous paragraph,  $(x_j - \sigma_j x_j)(\tau_1) =0$ for each $j$. It follows that $(\sigma x)(\tau_1) = 0$, which is a contradiction. This completes the proof.
\end{proof}

\begin{lemma}
	\label{lem:derived-saturated-submodule}
	Let $M$ be a derived saturated submodule of a semi-induced module $P$. Then $t_0(M) \le t_0(P)$.
\end{lemma}
\begin{proof} 	We proceed by induction on $d = t_0(P)$. Denote the induced module $P/ P_{\prec d}$ by $I$ and its submodule $(M + P_{\prec d})/P_{\prec d}$ by $N$. Suppose first that $N$ is an induced submodule of $I$. In this case, we have $t_0(N) \le t_0(I) = d$. Using the exact sequence \[ 0 \to M \cap P_{\prec d} \to M \to N \to 0 \] we see that $M \cap P_{\prec d}$ is a derived saturated submodule of $P_{\prec d}$. By induction, we have $t_0(M \cap P_{\prec d}) \le d-1$, and it follows that $t_0(M) \le d = t_0(P)$. Thus we can assume that $N$ is not an induced module. In this case, there exists an $r > d$ such that $\rH^{\VI}_0(N)_r$ is nonzero. Pick the least such $r$. We claim that $\rH^{\VI}_0(N)(\aF^r)$ is a $\bk[\GL(\aF^r)]$-submodule of $\cI(\rH^{\VI}_0(I/N)_d)(\aF^r)$.  To see this, let $N'$ be the submodule of $\cI(W)$ generated by $N_d$. By Proposition~\ref{prop:induced-generated-in-d}, we have $N' = \cI(N_d)$. By minimality of $r$, we have $N' = N_{\prec r} \subset N$. The claim now follows from the following: \begin{align*}
	\rH^{\VI}_0(N)(\aF^r) &= (N/N_{\prec r})(\aF^r)\\
	&= (N/\cI(N_d))(\aF^r) \\
	&\subset (I/\cI(N_d))(\aF^r) \\
	& = \cI((I/N)_d)(\aF^r) & \mbox{(By Proposition~\ref{prop:induced-from-d})}\\
	& = \cI(\rH^{\VI}_0(I/N)_d)(\aF^r).
	\end{align*}

	Let $A + B + X$ be a decomposition  of $\aF^r$ such that $\dim A = d$ and $\dim X = 1$. Set $W = \rH^{\VI}_0(I/N)(A)$. Clearly, $W$ is a $\bk[\GL(A)]$-module. Let $H$ be the subgroup of $\GL(\aF^r)$ that stabilizes $A$. There is a natural surjection $\phi \colon H \to \GL(A)$. We let $H$ act on $W$ via this surjection. Since $\bU_X(A + B)$ lies in the kernel of $\phi$, we see that $\bU_X(A+B)$ acts trivially on $W$. We also have \[\cI(\rH^{\VI}_0(I/N)_d)(\aF^r) = \Ind_{H}^{\GL(\aF^r)} W. \] By the previous lemma, we conclude that $(\newshift \rH^{\VI}_0(N))_{r-1}$ is nonzero. Since $\rH^{\VI}_0$ is right exact, it follows that $(\newshift \rH^{\VI}_0(M))_{r-1}$ is nonzero. By Proposition~\ref{prop:commuting-pairs}, we see that $t_0(\newdelta M) \ge r-1 > d-1$. But by Lemma~\ref{lem:new-saturated}, $\newdelta M$ is a derived saturated submodule of $\newdelta P$, which contradicts the inductive hypothesis. This contradiction completes the proof.
\end{proof}

The following argument is motivated by \cite[Proposition~2.9]{periodicity}.	

\begin{proposition}
	\label{prop:induced-resolution-semi-induced}
	Let $M$ be a module generated in finite degrees.   If $M$ is derived saturated then it admits a resolution $F_{\bullet} \to M$  of length at most $t_0(M) + 1$ where each $F_i$ is an induced module generated in finite degrees.
\end{proposition}
\begin{proof}	 Let $d = t_0(M)$, and let $r$ be the least number such that $\rH^{\VI}_0(M)$ is non-trivial in degree $r$. We prove  by induction on $d-r$ that there is a resolution $F_{\bullet} \to M$ of length at most $d - r+1$. Let $F_0 = \bigoplus_{0 \le k \le d} \cI(V_k)$ where $V_k = M_k$. We note that    $\rH^{\VI}_0(M)_r = V_r = \rH^{\VI}_0(F_0)_r$ and $\rH^{\VI}_0(M)_k = 0=\rH^{\VI}_0(F_0)_k $ for $k<r$. By construction, $t_0(F_0) \le d$ and there is a surjection $\psi \colon F_0 \to M$. Clearly, we have $\rH^{\VI}_0(\ker(\psi))_k =0$ for $k \le r$. Since both $M$ and $F_0$ are derived saturated, we see that $\ker(\psi)$ is derived saturated as well. By the previous lemma, $t_0(\ker(\psi)) \le d$. Thus by induction on $d-r$,  $\ker(\psi)$ admits a resolution of the desired format. We can append $F_0$ to this resolution to get a resolution of $M$, completing the proof.
\end{proof}

\begin{theorem}
	\label{thm:derived-saturated-are-semi-induced}
Assume that we are in the non-describing characteristic. Let $M$ be a module generated in finite degrees. Then $M$ is derived saturated  if and only if it is semi-induced.
\end{theorem}
\begin{proof}
	Corollary~\ref{cor:semi-induced-are-derived-saturated} shows that semi-induced modules are derived saturated. The other implication follows from the previous proposition and Corollary~\ref{cor:short-exact-sequence-semi-induced}.
\end{proof}

An $\FI$-module analog of the result above has been proven in \cite[Theorem~A.9]{djament}.

\subsection{The shift theorem}
Here we assume that $\bk$ is a noetherian ring. The following theorem independently proven by Putman-Sam \cite{putman-sam} and Sam-Snowden \cite{catgb}  is crucial for our purpose.

\begin{theorem}[{\cite{putman-sam, catgb}}]
	\label{thm:noetherianity}
	The category of $\VI$-modules over a noetherian ring is locally noetherian. In particular, if $M$ is a finitely generated $\VI$-module over $\bk$ then $\Gamma(M)$ is supported in finitely many degrees. 
\end{theorem}

We now state and prove our main theorem (an $\FI$-module analog has been proven by the author in \cite[Theorem~A]{nagpal}). 

\begin{theorem}[The shift theorem]
	\label{thm:shift-theorem}
	Assume that we are in the non-describing characteristic, and let $M$ be a finitely generated $\VI$-module. Then the following holds: \begin{enumerate}
	\item $\newshift^n M$ and $\Sigma^n M$ are semi-induced for large enough $n$.
	\item There exists a finite length complex $I^{\bullet}$ supported in non-negative degrees with the following properties \begin{itemize}
		\item $I^0 = M$.
		\item $I^i$ is semi-induced for $i > 0$.
		\item $I^n = 0$ for $n > t_0(M) + 1$.
		\item $\rH^i(I^{\bullet})$ is supported in finitely many degrees for each $i$.
	\end{itemize}
\end{enumerate}
\end{theorem}

We need a lemma.

\begin{lemma}
	Let  $Y$ be fixed vector space, and $N$ be a torsion $\VI$-module. For a vector space $X$, let $\ell_X$ denote the map from $0$ to $X$.  If $\newshift^Y\newshift^{\ell_X} \colon \newshift^Y N \to \newshift^Y \newshift^X N$ is an injection for all $X$ then $\newshift^Y N = 0$.
\end{lemma}
\begin{proof} Suppose, if possible, $\newshift^Y N(Z)$ is nontrivial for some vector space $Z$, and pick a nonzero element $x \in \newshift^Y N(Z)$. Let $\wt{x}$ be a lift of $x$ in $N(Y+Z)$. Since $N$ is torsion, there is a vector space $X$ such that for every linear injection $f \colon Y+Z \to X +Y+ Z$ the induced map $f_{\star} \colon N(Y+Z) \to N(Y+X+Z)$ takes $\wt{x}$ to zero. But this shows that $\newshift^Y\newshift^{\ell_X} $ takes $x$ to zero, contradicting the injectivity hypothesis. This completes the proof.
\end{proof}

\begin{proof}[Proof of Theorem~\ref{thm:shift-theorem}]
	We first prove that $\newshift^n M$ is semi-induced for large enough $n$. We do this by induction on $t_0(M)$. By Theorem~\ref{thm:noetherianity}, $h_0(M) < \infty$. Let $X$ be a non-trivial vector space. Then the cokernel $\newdelta^X M$ of $M \to \newshift^X M$ satisfies $t_0(\newdelta^X M) < t_0(M)$ (Corollary~\ref{cor:generation-of-newdelta}). Moreover, the kernel $K = \kappa^X(M)$ of $M \to \newshift^X M$ is a torsion-module supported in degrees $\le h_0(M)$ (Lemma~\ref{lem:newshift-injectivity}). 
	
	We claim that $\newshift^Y \newdelta^X M$ is semi-induced for large enough $Y$ which is independent of $\dim X$. To see this, suppose that $X$ is of dimension $g$. Since $t_0(\newdelta^X M) < t_0(M)$, the induction hypothesis implies that there exists a number $k_g$ such that $\newshift^Y \newdelta^X M$ is semi-induced whenever the dimension of $Y$ is larger than $k_g$. Pick a $t$ larger than $h_0(M)$ and $k_1$, and assume that the dimension of $Y$ is at least $t$.  Then $\newshift^Y K = \newshift^Y \kappa^X(M) = 0$, and so Lemma~\ref{lem:djament-composable-morphisms} yields the following exact sequence \[ 0 \to \newdelta^Y M \to \newdelta^{X+Y} M \to \newshift^Y \newdelta^X M \to 0. \] Now suppose $X$ is of dimension 1. Then the last term in this exact sequence is semi-induced as $t > k_1$. We conclude that $\newshift^{Y'} \newdelta^Y M$ is semi-induced if and only if $\newshift^{Y'} \newdelta^{Y+X} M$ is semi-induced (Corollary~\ref{cor:short-exact-sequence-semi-induced}). In other words, we may assume $k_{t+1} = k_t$ for any $t> \max(h_0(M), k_1)$.  Thus if $Y$ is of dimension larger than $h_0(M)$ and $k_i$ for $1 \le i \le \max(h_0(M), k_1) +1$, then $\newshift^Y \newdelta^X M$ is semi-induced for all $X$. This proves the claim.

	Let $Y$ be large enough such that $\newshift^Y \newdelta^X M$ is semi-induced for all $X$, and assume that the dimension of $Y$ is larger than $h_0(M)$. Then $\newshift^Y K = \newshift^Y \kappa^X(M) = 0$, and so we have an exact sequence: \[0 \to \newshift^Y M \to \newshift^Y \newshift^X M \to \newshift^Y \newdelta^X M \to 0.  \]   By Corollary~\ref{cor:semi-induced-are-derived-saturated} and Proposition~\ref{prop:triangle}, we see that $\rR \Gamma (\newshift^Y \newdelta^X M) = 0$. Thus by the exact sequence above, we conclude that $\rR^i \Gamma(\newshift^Y M) \cong \rR^i \Gamma(\newshift^Y \newshift^X M)$ where the isomorphism is given by $\rR^i \Gamma(\newshift^Y \newshift^\ell)$ where $\ell \colon 0 \to X$ is the unique map. By Proposition~\ref{prop:newshift-commutes-with-Gamma}, we see that \[\rR^i \Gamma(\newshift^Y \newshift^\ell -) = \newshift^Y \newshift^\ell  \rR^i\Gamma( - ).\] 	This shows that the map \[\newshift^Y \newshift^{\ell} \colon \newshift^Y \rR^i \Gamma(M) \to  \newshift^Y \newshift^X \rR^i\Gamma(M) \] is an isomorphism for each $X$. The previous lemma implies that $\newshift^Y \rR^i \Gamma(M) =0$. Thus $\rR^i\Gamma(  \newshift^Y M) =0$ for all $i$ (Proposition~\ref{prop:newshift-commutes-with-Gamma}). By Proposition~\ref{prop:triangle} and Theorem~\ref{thm:derived-saturated-are-semi-induced}, $\newshift^Y M$ is semi-induced. Thus $\newshift^n M$ is semi-induced for large $n$ (see Proposition~\ref{prop:iterated-shift}).

To prove that $\Sigma^n M$ is semi-induced for large enough $n$ we need part (b), which we now prove by induction on $t_0(M)$. Let $Y$ be a vector space such that $\newshift^Y M$ is semi-induced, and $\ell \colon 0 \to Y$ be the unique map. Set $I^0 = M$, $I^1 = \newshift^Y M$ where the map $I^0 \to I^1$ is $\newshift^{\ell}$. The cokernel of this map is $\newdelta^{\ell} M$. We have $t_0(\newdelta^{\ell} M) < t_0(M)$ (Corollary~\ref{cor:generation-of-newdelta}). By induction, there is a complex $J^{\bullet}$ of length at most $t_0(M)$ with $J^0 =  \newdelta^{\ell} M$, $J^i$ semi-induces for $i > 0$, and $\rH^i(J^{\bullet})$ finitely supported for each $i$. Now set $I^{i} = J^{i-1}$ for $i \ge 2$, and note that we can naturally append these to $I^0 \to I^1$ to get a complex $I^{\bullet}$. Clearly, this $I^{\bullet}$ has all the required properties. This proves part (b). 
	
Finally, we show that $\Sigma^n M$ is semi-induced for large enough $n$. For this let $I^{\bullet}$ be the complex as in part (b). Let $n$ be large enough such that $\deg \rH^i(I^{\bullet}) < n$ for all $i$. By construction, $\Sigma^n I^{\bullet}$ is exact and $\Sigma^n I^{i}$ are semi-induced for $i>0$ (shift of a semi-induced module is semi-induced; Corollary~\ref{cor:shift-of-semi-induced-and-presentation}). By Corollary~\ref{cor:short-exact-sequence-semi-induced}, $\Sigma^n I^0 = \Sigma^n M$ is semi-induced. This completes the proof.  
\end{proof}

\begin{remark}
	\label{rem:functoriality-shift-theorem} The proof of part (b) above shows that if $M \to N$ is a map of finitely generated $\VI$-modules then we can find complexes $I^{\bullet}$ and $J^{\bullet}$ for $M$ and $N$ respectively (with all the properties as mentioned in part (b)) and a natural map $I^{\bullet} \to J^{\bullet}$ extending the map $M \to N$.
\end{remark}

\begin{remark}
It is easy to see that the shift theorem together with Corollary~\ref{cor:short-exact-sequence-semi-induced} imply that $\ModVI$ is locally noetherian. Since we have only used Corollary~\ref{intro-cor-noeth} in our proof, it follows that Theorem~\ref{intro:noetherianity} is equivalent to its corollary. 
\end{remark}

\section{Some consequences of the shift theorem}

\label{sec:consequences}

Unless otherwise mentioned, we assume that we are in the non-describing characteristic, and that $\bk$ is noetherian.

\subsection{Stable degree and the $q$-polynomiality of dimension}
We define the {\bf stable degree} of a $\VI$-module $M$, denoted $\delta(M)$, by \[\delta(M) \coloneq \inf_{n \ge 0} t_0(\newshift^n M). \]  This is an invariant associated to $\VI$-module with several useful properties that we prove below. An invariant with the same name, but for $\FI$-modules, is discussed in \cite[\S 2]{linearranges}.

\begin{proposition}
	\label{prop:stable-degree}
	Let $M$ be a finitely generated module. We have the following: \begin{enumerate}
		\item  If $M$ is semi-induced, then $\delta(M) = t_0(M)$.
		\item  $\delta(M)$ is the common value of $t_0(\newshift^n M)$ for $n \gg 0 $. 
		\item  $\delta(M)$ is the common value of $t_0(\Sigma^n M)$ for $n \gg 0 $. 
		\item  $\delta(M) = \delta(\Sigma^n M) = \delta(\newshift^n M)$ for any $n \ge 0$.
		\item $\delta(M) \le t_0(M)<\infty$.
		\item  If $ 0 \to L \to M \to N \to 0$ is a short exact sequence, $\delta(M) = \max (\delta(L), \delta(N)).$
		\item If $K$  is a subquotient of $M$,  $\delta(K) \le \delta(M).$
		\item If $T$ is a torsion submodule of $M$, then $\delta(M/T) = \delta(M)$.
		\item The cokernel $\newdelta^X M$ of the natural map $M \to \newshift^X M$  satisfies $\delta(\newdelta^X M) < \max(\delta(M), 0)$.
	\end{enumerate}
\end{proposition}
\begin{proof}
	Part (a): First suppose that $M=\cI(V)$ is induced. From the equalities $\newshift \cI(V) = \cI(V) \oplus \cI(\newshift V)$ (Proposition~\ref{prop:categorical-derivation}) and $t_0(\cI(V)) = \deg V$, we see that $\delta(M) = t_0(\newshift^n M) = t_0(M)$. Since induced modules are acyclic with respect to  $\rH^{\VI}_0$ (Proposition~\ref{prop:ramos-new}) and $\newshift$ is exact, we conclude that the result holds for semi-induced modules as well. 
	
	Parts (b), (c), (d) and (e): Since $t_0(\newshift^n M)$ is a decreasing function of $n$ (Proposition~\ref{prop:decreasing-function}), we see that $\delta(M) = \delta(\newshift^n M)$ for any $n$. By the shift theorem (Theorem~\ref{thm:shift-theorem}) and part (a), we conclude that $\delta(M)$ is the common value of  $t_0(\newshift^n M)$ for $n \gg 0 $. Let $a$ be large such that $\newshift^a M$ is semi-induced and $n$ be large such that $\Sigma^n M$ is semi-induced (use the shift theorem again). Then we have an injection $\Sigma^n M \to \Sigma^n \newshift^a M$. By Corollary~\ref{cor:short-exact-sequence-semi-induced}, Proposition~\ref{prop:ramos-new} and part (a), we see that $t_0(\Sigma^n M) \le t_0(\Sigma^n \newshift^a M)  =  \delta(\newshift^a M) = \delta(M)$. Conversely, since we also have $t_0(\Sigma^n M) \ge t_0(\newshift^n M)$, we see that part (c) holds. Part (d) follows from (b) and (c) once we note that $t_0(\Sigma^n M)$ and $t_0(\newshift^n M)$ are decreasing functions of $n$ (Proposition~\ref{prop:decreasing-function}). Part (e) is trivial from this discussion.
	
	Parts (f), (g) and (h): Choose $n$ large enough that  $\Sigma^n L$, $\Sigma^n M$, and $\Sigma^n N$ are semi-induced. Since semi-induced modules are homology-acyclic, we have a short exact sequence \[0\to \rH^{\VI}_0(\Sigma^n L)\to \rH^{\VI}_0(\Sigma^n M)\to \rH^{\VI}_0(\Sigma^n N)\to 0.\] Thus, $t_0(\Sigma^n M)=\max(t_0(\Sigma^n L),t_0(\Sigma^n L))$, which implies the claim in light of part (c).  Part (g) is a consequence of part (f). For part (h), note that $T$ is supported in finitely many degrees (Theorem~\ref{thm:noetherianity}). By part (d), $\delta(T) =0$. Part (h) now follows from Part (f).
	
Part (i): First suppose that $M$ is semi-induced. Then by Corollary~\ref{cor:short-exact-sequence-semi-induced}, $\newdelta^X M$ is semi-induced. By Corollary~\ref{cor:generation-of-newdelta},  we see that $t_0(\newdelta^X M) < t_0(M)$. By part (a), we conclude that $\delta(\newdelta^X M) < \delta(M)$. Thus the result holds for semi-induced modules. Now suppose
that M is not semi-induced. Let $Y$ be large so that $\newshift^Y M$ is semi-induced. We have an exact sequence \[0 \to M/\kappa^Y(M) \to \newshift^Y M \to \newdelta^Y M \to 0. \] Applying $\newdelta^X$, we obtain the following exact sequence: \[ (\rL_1 \newdelta^X)(\newdelta^Y M) \to \newdelta^X(M/\kappa^Y(M)) \to \newdelta^X \newshift^Y M \to \newdelta^X \newdelta^Y M \to 0. \] The first term of this sequence is torsion (Proposition~\ref{prop:delta-1}). Thus by parts (g) and (h), we see that \[\delta(\newdelta^X(M/\kappa^Y(M)) ) \le \delta(\newdelta^X \newshift^Y M) < \delta(\newshift^Y M) = \delta(M).\] Now consider the exact sequence \[  \newdelta^X \kappa^Y(M) \to \newdelta^X M \to \newdelta^X(M/\kappa^Y(M)) \to 0.  \] Since the first term is torsion, we conclude that $\delta(\newdelta^X M) = \delta(\newdelta^X(M/\kappa^Y(M)) ) < \delta(M)$. This completes the proof.
\end{proof}

\begin{corollary}
	\label{cor:shift-complex-degrees} Let $I^{\bullet}$ be the complex as in Theorem~\ref{thm:shift-theorem}. Then we may assume that $t_0(I^1) = \delta(M)$, and $t_0(I^{i}) \le \delta(M) - i \le t_0(M) - i$ for $i > 1$.
\end{corollary}
\begin{proof}  This follows from the construction of $I^{\bullet}$ and the properties of the stable degree.
\end{proof}

\begin{lemma}
	\label{lem:dimension}
Assume that $\bk$ is a field. Let $\cI(V)$ be a module induced from $d$. Then \[\dim_{\bk} \cI(V)(\aF^n) = \frac{(q^n - 1)(q^n - q)\cdots(q^n - q^{d-1})}{(q^d - 1)(q^d - q)\cdots(q^d - q^{d-1})}\dim_{\bk} V(\aF^d) \] for every $n \ge 0$. In particular, there is a polynomial $P \in \bQ[X]$ such that $\dim_{\bk} \cI(V)(\aF^n) = P(q^n)$ for every $n \ge 0$.
\end{lemma}
\begin{proof} This easily follows from the equality $\cI(V)(\aF^n) = \bk[\Hom_{\VI}(\aF^d, \aF^n)] \otimes_{\bk[\GL_d]} V(\aF^d)$.
\end{proof}

\begin{theorem}[$q$-polynomiality of dimension]
	\label{thm:q-polynomiality-of-dimension}
Assume that $\bk$ is a field. Let $M$ be a finitely generated $\VI$-module. Then there exists a polynomial $P$ of degree $\delta(M)$ such that $\dim_{\bk} M(\aF^n) = P(q^n)$ for $n \gg 0$.
\end{theorem}
\begin{proof}
	Let $a$ be large enough such that $N \coloneq \Sigma^a M$ is semi-induced. By Proposition~\ref{prop:stable-degree}, we have $t_0(N) = \delta(M)$. Set $d = \delta(M)$. By Corollary~\ref{cor:good-filtration-new}, $N_{\preceq i}/N_{\prec i}$ is induced from $i$, and $N_{\preceq d}/N_{\prec d}$ is nonzero. By the previous lemma, there exists a polynomial $P$ such that $\dim_{\bk} N(\aF^n) = P(q^n)$ for every $n \ge 0$. This shows that $\dim_{\bk} M(\aF^n) = P(q^{n-a})$ for $n \ge a$, completing the proof.
\end{proof}

\begin{remark} The least $a$ such that $\Sigma^a M$ is semi-induced is exactly equal to $h^{\max}(M) + 1$ where \[h^{\max}(M) = \max_{i \ge 0} h^i(M) \] is the maximum of all local cohomology degrees. This follows easily from Theorem~\ref{thm:derived-saturated-are-semi-induced}, and the fact that $\Gamma$ commutes with $\Sigma$. We shall prove in the next section that $h^i(M)=0$ for $i > \delta(M) + 1$. Thus in the proof above, we have $\dim_{\bk}M(\aF^n) = P(q^n)$ for $n > \max_{0 \le i \le \delta(M) + 1} h^i(M) = h^{\max}(M)$.
\end{remark}

\subsection{Finiteness of local cohomology and regularity} 
\label{subsec:finiteness-regularity}
Let $\rD$ be the full triangulated subcategory of the bounded derived category $\rD^{\rb}(\ModVI)$ consisting of those objects that are represented by finite complexes with finitely generated cohomologies. 

\begin{proposition} 
	\label{prop:D}
	Let $M$ be an object of $\rD$. Then \begin{enumerate}
		\item $\rR \Gamma (M)$ is in $\rD$ and can be represented by a finite complex of finitely generated torsion modules.
		
		\item $\rR \bS (M)$ is in $\rD$ and can be represented by a finite complex of finitely generated induced modules.
		
		\item $\rR^i \Gamma(M)$ is finitely generated for each $i$ and vanishes if $i \gg 0$. 
		\item There is an exact triangle \[\rR\Gamma(M) \to M \to \rR \bS(M) \to.  \]
	\end{enumerate}
\end{proposition}
\begin{proof}
	 By the shift theorem (Theorem~\ref{thm:shift-theorem}) and Remark~\ref{rem:functoriality-shift-theorem}, we have an exact triangle of the form \[T \to M \to F \to \] in $\rD$ such that $T$ is represented by a finite complex of finitely generated torsion modules and $F$ is represented by a finite complex of finitely generated semi-induced modules (see \cite[Lemma~2.3]{regthm} for more details). By Proposition~\ref{prop:induced-resolution-semi-induced}, $F$ is quasi-isomorphic to a finite complex of finitely generated induced modules. By Corollary~\ref{cor:torsion-have-no-saturation} and Proposition~\ref{prop:induced-is-torsion-acyclic}, we have $\rR \Gamma(T) \cong T$ and $\rR \Gamma(F) = 0$. Thus by applying $\rR \Gamma$ to the triangle above yields $T \cong \rR \Gamma(M)$. By Corollary~\ref{cor:torsion-have-no-saturation},  Theorem~\ref{thm:derived-saturated-are-semi-induced}, we see that $\rR \bS(T) = 0$ and $\rR \bS(F) \cong F$. Thus by applying $\rR \bS $ to the triangle above yields $\rR \bS(M) \cong F$. The proof is now complete by Proposition~\ref{prop:triangle}.
\end{proof}

%

The $\FI$-module analog of the theorem below has been studied in \cite{symc1}.

\begin{theorem}[Finiteness of local cohomology]
	\label{thm:finiteness-local-cohomology}
	Let $M$ be a finitely generated $\VI$-module. Then $\rR \Gamma(M)$ and $ \rR \bS(M)$ are objects of $\rD$ and are supported in non-negative degrees. Moreover, we have the following \begin{enumerate}
		\item $\rR^i \Gamma(M) = 0$ if   $i > \delta(M) + 1$.
		\item $\rR^i \bS(M) = 0$ if $i > \delta(M)$.
		\item We have an exact sequence $0 \to \Gamma(M) \to M \to \bS(M) \to \rR^1 \Gamma(M) \to 0$.
		\item $\rR^{i+1} \Gamma(M) \cong \rR^i \bS(M)$ for $i \ge 1$. 
	\end{enumerate} 
\end{theorem} 
\begin{proof}Let $I = I^{\bullet}$ be the complex as in the shift theorem (Theorem~\ref{thm:shift-theorem}). Then $I$ is supported in non-negative degrees and $I^i = 0$ if $i > \delta(M) + 1$ (see Proposition~\ref{prop:stable-degree} part (i) and the construction of $I^{\bullet}$). We may take $T$, as in the proof of Proposition~\ref{prop:D}, to be equal (or quasi-isomorphic; see \cite[Lemma~2.3]{regthm}) to $I$. This shows that part (a) holds. The rest is immediate from Proposition~\ref{prop:triangle}.	
\end{proof}

\begin{corollary}
	\label{cor:shift-complex}
	Let $I^{\bullet}$ be the complex as in Theorem~\ref{thm:shift-theorem}. Then $\rR^i \Gamma(M) = \rH^i(I^{\bullet})$.
\end{corollary}

\begin{lemma}
	\label{lem:steinberg-complex}
	There is a resolution of the $\VI$-module $\bk = \bA/\bA_+$ of the form $\cI(\St_{\bullet})  \to \bk \to 0$, where $\St_{d}$ denote  the Steinberg representation of $\GL_d$.
\end{lemma}
\begin{proof}We refer the reader to \cite[pg 7]{charney} where an argument for split Steinberg representation is given. The argument for the Steinberg representation is similar. 
\end{proof}

\begin{lemma}
	\label{lem:regularity-of-A-plus}
	Let $M$ be a finitely generated torsion module, and suppose $\deg M = d$. Then $t_i(M) - i \le d$ for all $i \ge 0$. 
\end{lemma}
\begin{proof}
Since induced modules are homology-acyclic (Proposition~\ref{prop:ramos-new}), the previous lemma implies that	$\rH^{\VI}_i(M) = \Tor_i^{\bA}(\bk,  M) = \rH_i(\cI(\St_{\bullet})\otimes_{\bA} M)$. Clearly, $\cI(\St_{i})\otimes_{\bA} M = \St_i \otimes_{\VB} M$ is supported in degrees $\le d + i$. The result follows.
\end{proof}

For a finitely generated $\VI$-module $M$, let $r(M) = \max_{0 \le i \le \delta(M) + 1}(h^i(M) +i)$. The following argument is based on \cite[Corollary~2.5]{regthm}.

\begin{theorem}[Finiteness of regularity]
	\label{thm:regularity} Let $M$ be a finitely generated $\VI$-module. Then $t_i(M) -i \le r(M)$ for all $i >0$. In particular, $M$ has finite Castelnuovo-Mumford regularity.
\end{theorem}
\begin{proof}
By Theorem~\ref{thm:finiteness-local-cohomology} and the previous lemma, we see that $t_i(\rR \Gamma(M) ) - i \le r(M)$. Since $\rR \bS(M)$ is supported in non-negative cohomological degrees (which we think of as non-positive homological degrees), we conclude that $t_i(\rR \bS(M)) = 0$ for $i > 0$ (Proposition~\ref{prop:ramos-new}). The exact triangle $\rR \Gamma(M) \to M \to \rR \bS(M) \to $ of Proposition~\ref{prop:D} implies that $t_i(M) \le \max( t_i(\rR \Gamma(M) ), t_i(\rR \bS(M)) )$. Thus for $i > 0$, we obtain $t_i(M) - i \le r(M)$. This completes the proof.
\end{proof}

\subsection{Representation stability in characteristic zero} 
\label{sec:rep-stability}
In this section, we assume that $\bk$ is an algebraically closed field of characteristic 0. We first recall a parametrization of irreducible representations of $\GL_n$ over $\bk$, we follow \cite[\S 9]{zelevinsky}. Let $\cC_n$ be the isomorphism classes of cuspidal representations (irreducible representations which cannot be obtained via a parabolic induction) of $\GL_n$, and set $\cC = \sqcup_{n \ge 1} \cC_n$. If $\rho \in \cC_n$, we set $|\rho| = n$. Let $\cP$ be the set of partitions. Given a partition  $\lambda$, we set $|\lambda| = n$ if $\lambda$ is a partition of $n$. Given a function $\bmu \colon \cC \to \cP$, we set $|\bmu| = \sum_{x \in \cC} |x| |\bmu(x)|$.  The isomorphism classes of irreducible representations of $\GL_n$ are in bijection with the set of functions $\bmu$ satisfying $|\bmu| = n$. We fix an irreducible representation $M_{\bmu}$ corresponding to each partition function $\bmu$.

Let $\iota \in \cC_1$ be the trivial representation of $\GL_1$. For a partition function $\bmu$ with $\bmu(\iota) = \lambda$, we define another partition function $\bmu[n]$ by \[ \bmu[n](\rho) = \begin{cases}
(n - |\bmu|, \lambda_1, \lambda_2, \ldots) &\mbox{if } \rho = \iota \\
\bmu(\rho) &\mbox{if } \rho \neq \iota.
\end{cases} \] This definition makes sense only if $n \ge |\bmu| + \lambda_1$. 

Let \[M_0 \xrightarrow{\phi_0} M_1 \xrightarrow{\phi_1} M_2 \xrightarrow{\phi_2}  \cdots \] be a sequence such that each $M_n$ is a $\bk[\GL_n]$-module and each $\phi_n$ is $\GL_n$-equivariant. Following \cite{gan-watterlond-VI} which, in turn, is based on \cite{rep-stability}, we  call such a sequence {\bf representation stable of degree $d$ starting at $N$} if the following three conditions are satisfied for every $n \ge N$: \begin{enumerate}
	\item[(RS1)] {\bf Injectivity}: The map $\phi_n \colon M_n \to M_{n+1}$ is injective.
	\item[(RS2)] {\bf Surjectivity}: The $\GL_{n+1}$ orbits of $\phi_n(M_n)$ span all of $M_{n+1}$.
	\item[(RS3)] {\bf Multiplicities}: There is a decomposition \[M_n = \bigoplus_{\bmu} M_{\bmu[n]}^{\oplus c(\bmu)}  \] where the multiplicities $0 \le c(\bmu) < \infty$ do not depend on $n$, and $c(\bmu) = 0$ if $|\bmu| >d$.	
\end{enumerate} 

We now prove and improve \cite[Theorem~1.6]{gan-watterlond-VI}. 

\begin{theorem}[Representation stability]
	\label{thm:rep-stability}
	Let $M$ be a finitely generated $\VI$-module. Denote $M(\aF^n)$ by $M_n$, and let $\phi_n \colon M_n \to M_{n+1}$ be the map induced by the natural inclusion $\aF^n \hookrightarrow \aF^{n+1}$. Then the sequence \[M_0 \xrightarrow{\phi_0} M_1 \xrightarrow{\phi_1} M_2 \xrightarrow{\phi_2}  \cdots \] is representation stable of degree $\delta(M)$ starting at $N \coloneq \max(h^{\max}(M) +1, 2t_0(M)) $.
\end{theorem}
\begin{proof}
	Since $h^0(M) < N$, we see that (RS1) holds. Similarly, $t_0(M) \le N$ implies that (RS2) holds. Now we prove (RS3). Let $I^{\bullet}$ be the complex as in Theorem~\ref{thm:shift-theorem}. Then $I^{\bullet}(\aF^n)$ is exact if $n > h^{\max}(M)$ (Corollary~\ref{cor:shift-complex}). Since $I^0 = M$, it suffices to prove (RS3) for $I^i$ for each $i > 0$. We may also assume that $t_0(I^1) = \delta(M)$, and $t_0(I^{i}) \le \delta(M) - i \le t_0(M) - i$ for $i > 1$ (Corollary~\ref{cor:shift-complex-degrees}). Thus it suffices to show (RS3) for a semi-induced module generated in degrees $\le \delta(M)$. By Proposition~\ref{prop:projectives}, every semi-induced module is induced in characteristic zero. Thus we are reduced to showing (RS3) for a finitely generated induced module generated in degrees $\le \delta(M)$. This follows from Pieri's formula (see \cite[Lemma~2.8]{gan-watterlond-VI}), completing the proof.
\end{proof}

\subsection{Classification of indecomposable injectives in characteristic zero}

We first classify torsion-free injectives in the proposition below. We repeatedly use the fact that in characteristic zero, every induced module is projective (Proposition~\ref{prop:projectives}), and so every semi-induced module is, in fact, induced.

\begin{proposition}
	\label{prop:semi-induced-are-injectives-char0} Every induced (and hence semi-induced) $\VI$-module is injective in $\ModVI$. A torsion-free injective $\VI$-module is induced.
\end{proposition}
\begin{proof} Let $\cI(W)$ be a finitely generated induced module. Note that $\VI$-modules form a locally noetherian category (Theorem~\ref{thm:noetherianity}), and so any direct sum of injective modules is injective. Since any induced module is a direct sum of finitely generated induced modules, it suffices to show that  $\cI(W)$ is injective. 
	
We start by showing that $\Ext^1(Q, \cI(W)) =0$ for any finitely generated module $Q$. This is equivalent to showing that any short exact sequence of the form \[ 0 \to \cI(W) \to M \to Q \to 0 \] splits.  Thus it suffices to show that any injection $f \colon \cI(W) \to M$ admits a section whenever $M$ is finitely generated. Let $X$ be a vector space of large enough dimension so that $\Sigma^X M$ is semi-induced (Theorem~\ref{thm:shift-theorem}). Let $\ell \colon 0 \to X$ be the unique map. Exactness of $\Sigma^X$ and the commutativity of the diagram below 
	\[
	\begin{tikzcd}
	\Sigma^X \cI(W) \ar{r}{\Sigma^{X}(f)} &  \Sigma^X M\\ 
	\cI(W) \ar{u}{\Sigma^{\ell}} \ar{r}{f} & M \ar{u}{\Sigma^{\ell}} 
	\end{tikzcd}
	\]
	shows that $\Sigma^{\ell} f \colon \cI(W) \to \Sigma^{X} M$ is injective. By Corollary~\ref{cor:short-exact-sequence-semi-induced}, the cokernel of $\Sigma^{\ell} f$ is semi-induced. By Proposition~\ref{prop:projectives} and the characteristic 0 hypothesis, every semi-induced module is projective. Hence $\Sigma^{\ell} f$ admits a section $s$. Then $s\Sigma^{\ell}$ is a section of $f$, as required.

Next, let  $M \subset N$ be arbitrary $\VI$-modules, and $\Phi \colon M \to \cI(W)$ be an arbitrary map. We will show that $\Phi$ extends to $N$ which finishes the proof of the first assertion. We follow the proof of  Baer's criterion as in \stacks{0AVF}. By Zorn's lemma, it suffices to show that if $M \subsetneq N$ then $\Phi$ extends to a submodule $M' \subset N$ which properly contains $M$. For this, pick an $x \in N \setminus M$, and let $M'$ be the submodule of $N$ generated by $M$ and $x$. Then $x \in N(\aF^d)$ for some $d$. Note that \[Q \coloneq \{f \in \cI(d) \colon fx \in M \}\] is a $\VI$-submodule of $\cI(d)$. By the previous paragraph, we have $\Ext^1(\cI(d)/Q, \cI(W)) =0$. Thus the map $\psi \colon Q \to \cI(W)$ given by $f \mapsto \Phi(fx)$  extends to a map $\wt{\psi} \colon \cI(d) \to \cI(W)$. Now consider the map $\wt{\Psi} \colon M \oplus \cI(d) \to \cI(W)$ given by \[ (y, f) \mapsto \Phi(y)- \wt	{\psi}(f). \] The kernel of this map contains the kernel of the natural map $M \oplus \cI(d) \to N$ given by $(y, f) \mapsto y + fx$. Thus $\wt{\Psi}$ factors through a map $\Psi \colon M' \to \cI(W)$. It is easy to check that this map extends $\Phi$. This concludes the proof of the first assertion.

	Let $I$ be an arbitrary torsion-free injective module. Then by the shift theorem, $I$ embeds into a direct sum $J$ of induced modules. Since $I$ is injective, the embedding $I \to J$ splits. This shows that the injection $I_{\preceq d} \to J_{\preceq d}$ is split as well, and so $I_{\preceq d}$ is injective and torsion-free. It follows that $\rR \Gamma(I_{\preceq d}) = 0$, and so $I_{\preceq d}$ must be derived saturated. Thus $I_{\preceq d}$ is induced (Theorem~\ref{thm:derived-saturated-are-semi-induced}). since colimits are exact and $I = \varinjlim_d I_{\preceq d}$, we see that $I$ is a direct sum of induced modules, concluding the proof of the second assertion. 
\end{proof}

We now classify torsion injectives. For this we do not need any assumption on $\bk$ (noetherianity is still needed but the non-describing characteristic assumption is not needed).  So assume that $\bk$ is an arbitrary noetherian ring. Let $\cV$ be a monoidal category. Given two functors $F_1 \colon \cC \to \cV$ and  $F_2 \colon \cC^{\op} \to \cV$  there is a natural notion of a tensor product $F_1 \otimes_{\cC} F_2$ (we refer the readers to  \cite{tensor-hom} for details). More explicitly, if $\cC = \VI$ and $\cV = (\Mod_{\bk}, \otimes_{\bk})$, then $F_1 \otimes_{\VI} F_2$ is given by the following $\bk$-module \[ \left( \bigoplus_{X \in \Obj(\VI) }	F_1(X) \otimes_{\bk} F_2(X) \right)/  \langle f_{\star}(v) \otimes w - v \otimes f^{\star}(w) \colon f \in \Hom_{\VI}(X,Y), v \in F_1(X), w \in F_2(Y) \rangle.  \]

The following lemma is elementary.

\begin{lemma}
	\label{lem:ideal-tensor}
	$\bk[\Hom_{\VI}(-, \aF^d)]$ is a projective $\VI^{\op}$-module. Moreover, for any $\VI$-module $N$, we have \[\bk[\Hom_{\VI}(-, \aF^d)] \otimes_{\VI} N = N(\aF^d).\]
\end{lemma}

Let $E$ be a $\bk[\GL_d]$-module. We denote by $\dualI(E)$ the $\VI$-module given by \[\Hom_{\bk[\GL_d]}(\bk[\Hom_{\VI}(-, \aF^d)], E).\] $\dualI(E)$ is clearly a torsion $\VI$-module (note that $\dualI(E)(Y) = 0$ for $Y \succ \aF^d$).

\begin{proposition}
	\label{prop:injective-adjunctions}
	For any  $\bk[\GL_d]$-module, we have \[\Hom_{\ModVI}(M, \dualI(E)) = \Hom_{\bk[\GL_d]}(M(\aF^d), E).\] In particular, if $E$ is an injective $\bk[\GL_d]$-module then $\dualI(E)$ is an injective $\VI$-module.
\end{proposition}
\begin{proof} By the tensor-hom adjunction, we have \begin{displaymath}
	\begin{split}
	\Hom_{\ModVI}(M , \dualI(E)) &= \Hom_{\ModVI}(M , \Hom_{\bk[\GL_d]}(\bk[\Hom_{\VI}(-, \aF^d)], E))\\
	&= \Hom_{\bk[\GL_d]}(\bk[\Hom_{\VI}(-, \aF^d)] \otimes_{\VI} M , E)\\
	&= \Hom_{\bk[\GL_d]}(M(\aF^d), E)
	\end{split}
	\end{displaymath} where the last equality follows from the previous lemma. If $E$ is injective, the functor given by \[M \mapsto  \Hom_{\bk[\GL_d]}(M(\aF^d) , E)\] is exact, and hence $\dualI(E)$ is injective. 
\end{proof}

For a $\VI$-module $M$, we denote the maximal submodule supported in degrees $\le d$ by $M^{\preceq d}$.

\begin{proposition} 
	\label{prop:d-th-part-is-injective}
	Suppose $I = \dualI(E)$. Then $I(\aF^d) \cong E$. Moreover, $I^{\prec d} = 0$ and $I^{\preceq d} = I$.
\end{proposition}
\begin{proof} Clearly, $I(\aF^d) = \Hom_{\bk[\GL_d]}(\bk[\GL_d], E) \cong E$. For the second statement, it suffices to show that if $\Psi$ is a nonzero element of $I(X)$, then $g_{\star}(\Psi)$ is nonzero for any $g \in \Hom_{\VI}(X, Y)$ with $Y \preceq \aF^d$. So suppose $\Psi \in I(X) =  \Hom_{\bk[\GL_d]}(\bk[\Hom_{\VI}(X, \aF^d)], E)$. If $\Psi$ is nonzero then there exists an $h \in \Hom_{\VI}(X,\aF^d)$ such that $\Psi(h) \neq 0$. Let $f \in \Hom_{\VI}(X, \aF^d)$ be such that $fg = h$. Now $(g_{\star} (\Psi))(f) = \Psi(fg) = \Psi(h) \neq 0$. Thus $g_{\star} (\Psi)$ is nonzero completing the proof.
\end{proof}

A {\bf principal injective} of type $d$ is a $\VI$-module of the form $\dualI(E)$ where $E$ is an injective $\bk[\GL_d]$-module. By Proposition~\ref{prop:d-th-part-is-injective}, the degree $d$ part of a principal injective of type $d$ is an injective $\bk[\GL_d]$-module.

\begin{lemma} 
	\label{lem:embed-into-principal-injective}
	Let $M$ be a $\VI$-module. Then $M^{\preceq d}/M^{\prec d}$ injects into a principal injective $I$ of type $d$. In fact, if $E$ is the injective hull of $M^{\preceq d}(\aF^d)$ as a $\bk[\GL_d]$-module, then we may take $I = \dualI(E)$. 
\end{lemma}
\begin{proof} Let $N = M^{\preceq d}/M^{\prec d}$. Then $N$ is supported in degree $\le d$, and by definition of $N$,  $\ell_{\star} \colon N(X) \to N(\aF^d)$ is injective for any $X$ and any $\ell \in \Hom_{\VI}(X, \aF^d)$. Thus if $f \colon N \to I$ is a map, then $f$ is injective if and only if $f(\aF^d) \colon N(\aF^d) \to I(\aF^d)$ is injective. Now let $\iota \colon N(\aF^d) \to E$ be the injective-hull of $N(\aF^d) = M^{\preceq d} (\aF^d)$ as a $\bk[\GL_d]$-module. Then by Proposition~\ref{prop:injective-adjunctions}, $\iota$ induces a map $\iota_{\star} \colon N \to \dualI(E)$. By our previous argument, it suffices to show that it is injective in degree $d$. But in degree $d$, this map is given by the image of $\iota$ under the natural adjunction isomorphism $\Hom_{\bk[\GL_d]}(N(\aF^d), E) \to  \Hom_{\bk[\GL_d]}(N(\aF^d), \dualI(E)(\aF^d))$ (see Proposition~\ref{prop:d-th-part-is-injective}) and hence is injective. 
\end{proof}

\begin{proposition} 
	\label{prop:torsion-fp-embeds-into-torsion-injective}
	Suppose $M$ is supported in degrees $\le d$. Let $E_k$ be the injective-hull of $M^{\preceq k}(\aF^k)$ as a $\bk[\GL_k]$-module. Then $M$ embeds into the injective module $\bigoplus_{k \le d} \dualI(E_k)$.
\end{proposition}
\begin{proof} If a module is supported in degree $\le d$, then it admits a filtration with modules of the form $M^{\preceq k}/M^{\prec k}$ with $k \le d$. The proposition now follows from Lemma~\ref{lem:embed-into-principal-injective} and the horseshoe lemma.
\end{proof}

\begin{proposition} 
A direct sum of injectives is injective. If $M$ is any torsion module and $E_k$ is the injective-hull of $M^{\preceq k}(\aF^k)$ as a $\bk[\GL_d]$-module, then $M$ embeds into the injective module $\bigoplus_{k \ge 0} \dualI(E_k)$.
\end{proposition}
\begin{proof}
	It is a standard fact that in a locally noetherian category a direct sum of injectives is an injective. Thus the first statement follows (Theorem~\ref{thm:noetherianity}). Now let $M$ be a torsion module. Then $M = \varinjlim_{d} M^{\preceq d}$ is a filtered colimit of modules supported in finitely many degrees. Since $\ModVI$ is a Grothendieck category, filtered colimits are exact. Hence the result follows from Proposition~\ref{prop:torsion-fp-embeds-into-torsion-injective}.
\end{proof}

\begin{proposition} 
	\label{prop:classification-of-torsion-injectives}
	A torsion module is injective in $\ModVI^{\tors}$ if and only if it is isomorphic to a direct sum of principal injectives. In particular, a torsion module is injective in $\ModVI^{\tors}$ if and only if it is injective in $\ModVI$. 
\end{proposition}
\begin{proof} By the previous proposition, a direct sum of principal injectives is injective. Let $I$ be a torsion injective. Then by the previous proposition again, $I$ admits an embedding $f \colon I \to J \coloneq  \bigoplus_{k \ge 0} \dualI(E_k)$ where $E_k$ is the injective-hull of $I^{\preceq k}(\aF^k)$ as a $\bk[\GL_k]$-module. Since $I$ is injective in $\ModVI^{\tors}$, $f$ admits a section $s$. This implies that $I^{\preceq k}/I^{\prec k}$ is a direct summand of $J^{\preceq k}/J^{\prec k} = \dualI(E_k)$. Thus $(I^{\preceq k}/I^{\prec k})(\aF^k) = I^{\preceq k}(\aF^k)$ is a direct summand of $\dualI(E_k)(\aF^k) = E_k$. Since a direct summand of injective module is injective, we see that $I^{\preceq k}(\aF^k)$ is injective, and hence is equal to its injective hull $E_k$. Thus if $K = \coker(f)$, then $(K^{\preceq k}/K^{\prec k})(\aF^k) = 0$ for each $k$. By the nakayama lemma, $K=0$. This shows that $f$ is an isomorphism, completing the proof.
\end{proof}

We are now ready to prove our main theorem on classification of indecomposable injectives. Note that the $\FI$-module analog of this result is proved in \cite[Theorem~4.3.4]{symc1}.

\begin{theorem}[Classification of indecomposable injectives] Assume that $\bk$ is a field of characteristic zero. Every injective is a direct sum of a torsion-free  injective and a torsion injective. Moreover, we have the following: \begin{enumerate}
		\item 	The set of torsion-free indecomposable injectives consists of modules of the form $\cI(E)$ where $E$ (or, more precisely, $E(\aF^d)$) is an irreducible $\bk[\GL_d]$-module for some $d$.
		
			\item 	The set of torsion indecomposable injectives consists of modules of the form $\dualI(E)$ where $E$ is an irreducible  $\bk[\GL_d]$-module for some $d$.
	\end{enumerate}
\end{theorem}
\begin{proof}
	In light of Lemma~\ref{lem:torsion-injective}, every injective is a direct sum of a torsion injective and a torsion-free injective. Part (a) follows from Proposition~\ref{prop:semi-induced-are-injectives-char0}, and part (b) follows from Proposition~\ref{prop:classification-of-torsion-injectives}.
\end{proof}

\subsection{Finiteness of injective dimension in characteristic zero}

\begin{lemma}
	Let $M$ be a finitely generated torsion module. Then $M$ has finite injective dimension.
\end{lemma}
\begin{proof}
	We prove the assertion by induction on $d = h^0(M)$. We have an exact sequence \[0 \to M^{\prec d} \to M \to M^{\preceq d}/M^{\prec d} \to 0. \] Since $h^0(M^{\prec d} ) <d$, the induction hypothesis implies that $ M^{\prec d}$ has finite injective dimension. By the horseshoe lemma, it suffices to prove that $M^{\preceq d}/M^{\prec d}$ has finite injective dimension. For that, let $E = M^{\preceq d}/M^{\prec d}(\aF^d)$. Since we are in characteristic zero, $E$ is an injective $\bk[\GL_d]$-module. As in Lemma~\ref{lem:embed-into-principal-injective}, there is an embedding $\iota \colon M^{\preceq d}/M^{\prec d} \to \dualI(E)$ which induces an isomorphism in degree $d$. This shows that $h^0(\coker (\iota)) < d$. By induction, $\coker (\iota)$ has finite injective dimension. Since $\dualI(E)$ is injective, we conclude that $M^{\preceq d}/M^{\prec d} $ has finite injective dimension, concluding the proof.
\end{proof}

The $\FI$-module analog of the following result is proved in \cite[Theorem~4.3.1]{symc1}.

\begin{theorem}[Finiteness of injective dimension]
	Every finitely generated $\VI$-module has finite injective dimension.
\end{theorem}
\begin{proof}
	Let $M$ be a finitely generated $\VI$-module. By Proposition~\ref{prop:D}, there exists an exact triangle \[X \to M \to F \to \] where $X$ is a finite length complex of finitely generated torsion modules and $F$ is a finite length complex of finitely generated semi-induced modules. In characteristic zero, every semi-induced module is injective. Thus is suffices to show that every finitely generated torsion module has finite injective dimension. But this is the content of the previous lemma. This finishes the proof.	
\end{proof}

\end{document}